\documentclass[10pt]{amsart}
\usepackage{amssymb,amsmath,amsthm,amscd}
\usepackage{leftidx}
\usepackage{graphicx}
\usepackage{color}
\usepackage{hyperref}
\usepackage[usenames,dvipsnames,svgnames,table]{xcolor}
\usepackage{tikz}
\usepackage{comment}
\usepackage{epstopdf} 
\usepackage{float}
\usepackage{tikz-cd}
\newtheorem{lemma}{Lemma}[section]
\newtheorem{theorem}[lemma]{Theorem}
\newtheorem{corollary}[lemma]{Corollary}
\newtheorem{proposition}[lemma]{Proposition}
\newtheorem{conjecture}[lemma]{Conjecture}
\newtheorem{question}[lemma]{Question}
\theoremstyle{definition}
\newtheorem{definition}[lemma]{Definition}
\theoremstyle{remark}
\newtheorem{remark}[lemma]{Remark}

\newcommand{\Q}{\mathbb{Q}}
\newcommand{\R}{\mathbb{R}}
\newcommand{\Z}{\mathbb{Z}}

\renewcommand{\epsilon}{\varepsilon}
\renewcommand{\phi}{\varphi}
\renewcommand{\theta}{\vartheta}

\DeclareMathOperator{\Deg}{deg}

\DeclareMathOperator{\ind}{ind}
\DeclareMathOperator{\Poly}{Poly}

\title[Discontinuity of Straightening]{Discontinuity of Straightening in Anti-holomorphic Dynamics: I\\}

\author[H.~Inou]{Hiroyuki Inou}
\address{Department of Mathematics, Kyoto University, Kyoto 606-8502, Japan}
\email{inou@math.kyoto-u.ac.jp}

\author[S.~Mukherjee]{Sabyasachi Mukherjee}
\address{School of Mathematics, Tata Institute of Fundamental Research, 1 Homi Bhabha Road, Mumbai 400005, India}
\email{sabya@math.tifr.res.in} 


\date{\today}

\begin{document}

\begin{abstract}
It is well known that baby Mandelbrot sets are homeomorphic to the original one. We study baby Tricorns appearing in the Tricorn, which is the connectedness locus of quadratic anti-holomorphic polynomials, and show that the dynamically natural straightening map from a baby Tricorn to the original Tricorn is discontinuous at infinitely many explicit parameters. This is the first known example of discontinuity of straightening maps on a real two-dimensional slice of an analytic family of holomorphic polynomials. The proof of discontinuity is carried out by showing that all non-real \emph{umbilical cords} of the Tricorn wiggle, which settles a conjecture made by various people including Hubbard, Milnor, and Schleicher. 
\end{abstract}

\maketitle

\setcounter{tocdepth}{1}

\tableofcontents

\section{Introduction}
Renormalization is one of the most powerful tools in the study of dynamical systems. In the celebrated paper \cite{DH2}, Douady and Hubbard developed the theory of polynomial-like maps to study renormalizations of complex polynomials, and proved the \emph{straightening theorem} that allows one to study a sufficiently large iterate of a polynomial by associating a simpler dynamical system, namely a polynomial of smaller degree, to it. They used it to explain the existence of small homeomorphic copies of the Mandelbrot set in itself. The fact that baby Mandelbrot sets are homeomorphic to the original one, is in some sense, a strictly `no interaction among critical orbits' phenomenon. The existence of critical orbit interactions for higher degree polynomials allow for much more complicated dynamical configurations, and the corresponding straightening maps are typically not as well-behaved as in the unicritical case.

Substantial progress in understanding the combinatorics and topology of straightening maps for higher degree polynomials has been made by Epstein (manuscript), the first author and Kiwi \cite{IK, I} in recent years. While the situation in quadratic dynamics is extremely satisfactory where each baby Mandelbrot set is homeomorphic to the original one (even quasiconformally equivalent in some cases \cite{L5}), such a miracle cannot be expected in the parameter spaces of higher degree polynomials. The first author showed that straightening maps are typically discontinuous in the presence of critical orbit relations. The proof of discontinuity of straightening maps given in \cite{I}, however, makes essential use of two complex dimensional bifurcations, and can not be applied to one-parameter families. Thus, the question whether straightening maps could fail to be continuous in one-parameter families, remained open. In particular, it was conjectured that straightening maps for quadratic anti-holomorphic polynomials (which form a real two-dimensional slice of biquadratic polynomials) are discontinuous, provided that the renormalization period is odd. The main purpose of this paper is to prove this conjecture in complete generality. In fact, we prove discontinuity of straightening maps in every even degree unicritical anti-holomorphic polynomial family. 

The dynamics of quadratic anti-holomorphic polynomials and its connectedness locus, the Tricorn, was first studied in \cite{CHRS}, and their numerical experiments showed major structural differences between the Mandelbrot set and the Tricorn; in particular, they observed that there are bifurcations from the period $1$ hyperbolic component to period $2$ hyperbolic components along arcs in the Tricorn (see Figure~\ref{umbilical_wiggle}(Left)), in contrast to the fact that bifurcations are always attached at a single point in the Mandelbrot set. The bifurcation structure in the family of quadratic anti-holomorphic polynomials was studied in \cite{Wi}. However, it was Milnor who first observed the importance of the multicorns $\mathcal{M}_d^*$ (which are the connectedness loci of unicritical anti-holomorphic polynomials $\overline{z}^d+c$); he found little Tricorn and multicorn-like sets as prototypical objects in the parameter space of real cubic polynomials \cite{M3}, and in the real slices of rational maps with two critical points \cite{M4}. Nakane \cite{Na1} proved that the Tricorn is connected, in analogy to Douady and Hubbard's classical proof of connectedness of the Mandelbrot set. This generalizes naturally to multicorns of any degree. Later, Nakane and Schleicher, in \cite{NS}, studied the structure of hyperbolic components  of the multicorns via the multiplier map (even period case) and the critical value map (odd period case). These maps are branched coverings over the unit disk of degree $d-1$ and $d+1$ respectively, branched only over the origin. Hubbard and Schleicher \cite{HS} proved that the multicorns are not pathwise connected (see Figure~\ref{umbilical_wiggle}(Right)), confirming a conjecture of Milnor. Recently, in an attempt to explore the topological aspects of the parameter spaces of unicritical anti-holomorphic polynomials, the combinatorics of external dynamical rays of such maps were studied in \cite{Sa} in terms of orbit portraits, and this was used in \cite{MNS} where the bifurcation phenomena, boundaries of odd period hyperbolic components, and the combinatorics of parameter rays were described. The authors showed in \cite{IM} that many parameter rays of the multicorns non-trivially accumulate on persistently parabolic regions. 

\begin{figure}[ht!]
\includegraphics[scale=0.15]{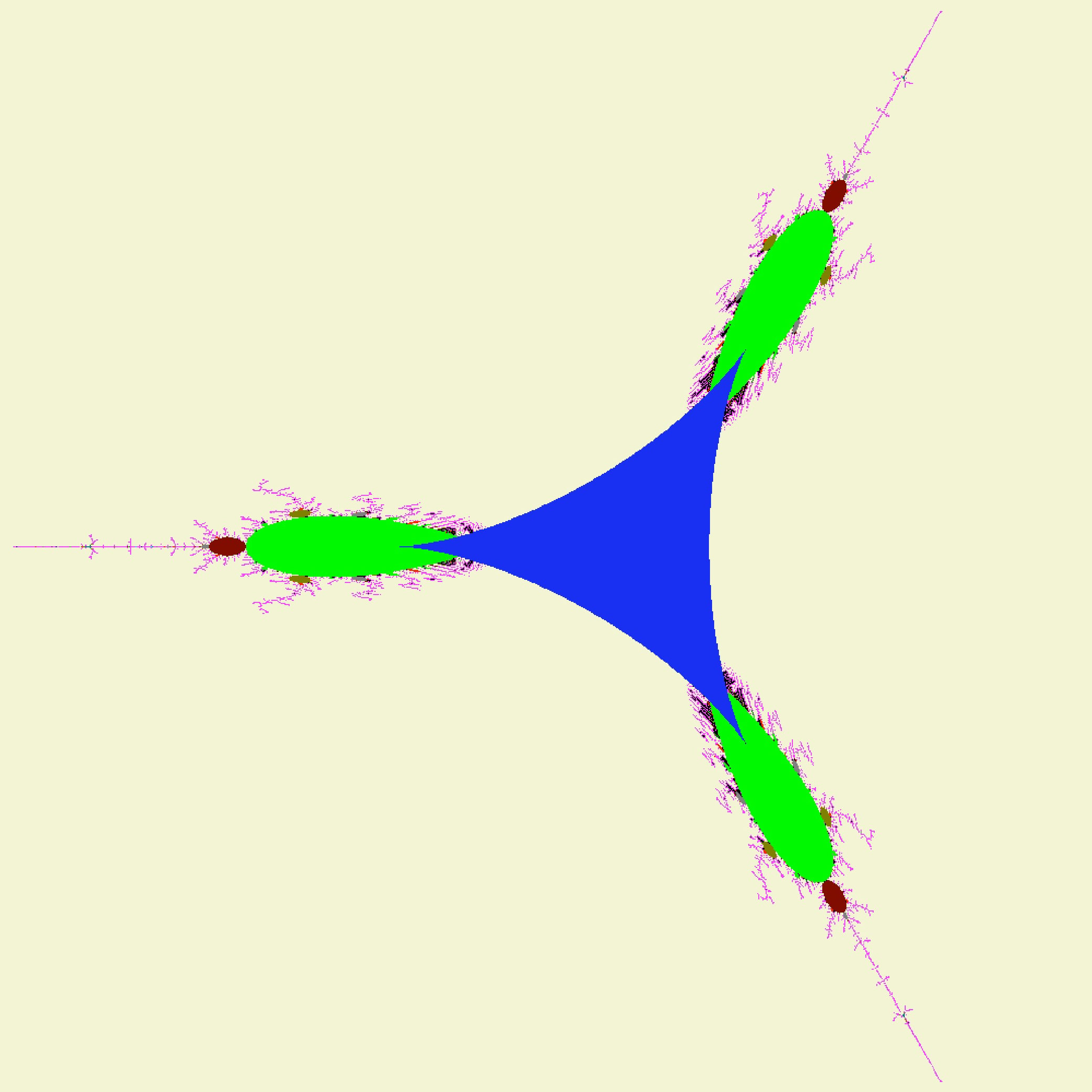} \hspace{1mm} \includegraphics[scale=0.25]{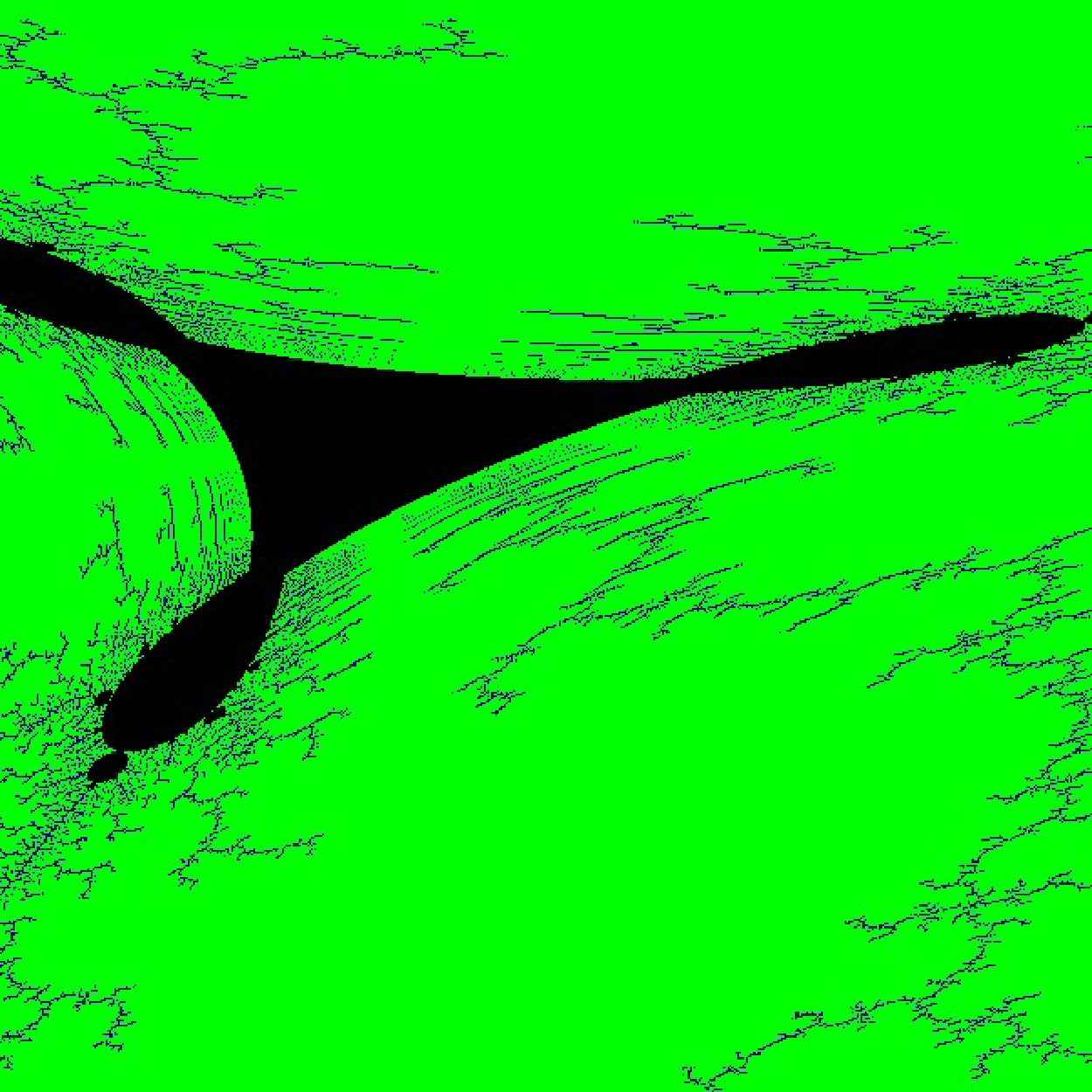}
\caption{Left: The Tricorn. Right: Wiggling of an umbilical cord on the root parabolic arc of a hyperbolic component of period $5$ of the Tricorn.}
\label{umbilical_wiggle}
\end{figure}

The multicorns can be thought of as objects of intermediate complexity between one dimensional and higher dimensional parameter spaces. Douady's famous `plough in the dynamical plane, and harvest in the parameter plane' principle continues to stand us in good stead since our parameter space is still real two-dimensional. But since the parameter dependence of anti-polynomials is only real-analytic, one cannot typically use complex analytic techniques to study the multicorns directly. This can be circumvented by passing to the second iterate, and embedding the family $\bar{z}^d +c$ in the family $\mathcal{F}_d = \lbrace (z^d+a)^d+b: a, b \in \mathbb{C} \rbrace$ of holomorphic polynomials. Thus the multicorn $\mathcal{M}_d^*$ is the intersection of the real $2$-dimensional slice $\lbrace a = \bar{b} \rbrace$ with the connectedness locus of $\mathcal{F}_d$, and hence it reflects several properties of higher dimensional parameter spaces. In fact, we will heavily exploit the critical orbit interactions of the polynomials $(z^d+a)^d+b$, and the existence of non-trivial deformation classes of parabolic parameters in the multicorns in our proof of discontinuity of straightening maps. 

The combinatorics and topology of the multicorns differ in many ways from those of their holomorphic counterparts, the multibrot sets, which are the connectedness loci of degree $d$ unicritical polynomials. At the level of combinatorics, this is manifested in the structure of orbit portraits \cite[Theorem~2.6, Theorem~3.1]{Sa}. The topological features of the multicorns have quite a few properties in common with the connectedness locus of real cubic polynomials, e.g.\ discontinuity of landing points of dynamical rays, bifurcation along arcs, existence of real-analytic curves containing q.c.-conjugate parabolic parameters, lack of local connectedness of the connectedness loci, non-landing stretching rays, etc. \cite{lavaurs_systemes_1989}, \cite{KN}, \cite[Corollary~3.7]{HS}, \cite{IM}, \cite[Theorem~3.2, Theorem~6.2]{MNS}, \cite{Sa1}. These are in stark contrast with the multibrot sets.

Numerical experiments suggest that every odd period hyperbolic component of the multicorns is the basis of a small `copy' of the multicorn itself, much like the corresponding phenomenon for the Mandelbrot set. While it is true that an anti-holomorphic analogue of the straightening theorem does provide us with a map from every small multicorn-like set to the original multicorn, it had been conjectured by various people, including Milnor, Hubbard, and Schleicher, that this map is discontinuous \cite{HS, MP}. The first author recently gave a computer-assisted proof of this fact for a particular candidate \cite{I1}. The principal goal of this paper is to prove this conjecture for every multicorn-like set contained in multicorns of even degree.

\begin{theorem}[Discontinuity of Straightening]\label{Straightening_discontinuity}
Let $d$ be even, $c_0$ be the center of a hyperbolic component $H$ of odd period (other than $1$) of $\mathcal{M}_d^*$, and $\mathcal{R}(c_0)$ be the corresponding $c_0$-renormalization locus. Then the straightening map $\chi_{c_0} : \mathcal{R}(c_0) \rightarrow \mathcal{M}_d^*$ is discontinuous (at infinitely many explicit parameters).
\end{theorem}

The proof of discontinuity is carried out by showing that the straightening map from a baby multicorn-like set to the original multicorn sends certain `wiggly' curves to landing curves. More precisely, for even degree multicorns, there exist hyperbolic components $H$ intersecting the real line, and their `umbilical cords' land on the root parabolic arc on $\partial H$. In other words, such a component can be connected to the period $1$ hyperbolic component by a path. However, we will prove that if $H$ does not intersect the real line or its rotates, then no path $\gamma$ contained in $\mathcal{M}_d^* \setminus \overline{H}$ can land on the root parabolic arc on $\partial H$ (this holds for multicorns of any degree). The non-existence of such a path will be referred to as the `wiggling' of `umbilical cords' of non-real hyperbolic components. Hence, for even degree multicorns, the (inverse of the) straightening map sends a piece of the real line to a `wiggly' curve; which is an obstruction to continuity.

The following theorem generalizes the main result of \cite{HS}, and shows that path-connectivity fails to hold in a very strong sense for the multicorns. We should mention that this is a major topological difference from the Mandelbrot set. In fact, any two Yoccoz parameters (i.e., at most finitely renormalizable parameters) in the Mandelbrot set can be connected by an arc in the Mandelbrot set \cite[Theorem~5.6]{S4}, \cite{PR}. 

\begin{theorem}[Umbilical Cord Wiggling]\label{umbilical_cord_wiggling}
Let $H$ be a hyperbolic component of odd period $k$ of $\mathcal{M}_d^*$, $\mathcal{C}$ be the root arc on $\partial H$, and $\widetilde{c}$ be the critical Ecalle height $0$ parameter on $\mathcal{C}$. If there is a path $p : [0,\delta] \rightarrow \mathbb{C}$ with $p(0) = \widetilde{c}$, and $p((0,\delta]) \subset \mathcal{M}_d^* \setminus \overline{H}$, then $d$ is even, and $\widetilde{c} \in \mathbb{R} \cup \omega\mathbb{R} \cup \omega^2\mathbb{R} \cup \cdots \cup \omega^{d}\mathbb{R}$, where $\omega=\exp(\frac{2\pi i}{d+1})$. 
\end{theorem}

The existence of non-landing umbilical cords for the multicorns was first proved by Hubbard and Schleicher \cite{HS} (see also \cite{NS1}) under a strong assumption of non-renormalizability. The main technical tool in their proof is the theory of perturbation of anti-holomorphic parabolic points \cite[\S 4]{HS}, \cite[\S 2]{IM}. Using these perturbation techniques, they showed that the landing of an umbilical cord at $\widetilde{c}$ implies that a loose parabolic tree of $f_{\widetilde{c}}$ would contain a real-analytic arc connecting two bounded Fatou components. With the assumption of non-renormalizability, one can deduce from the above statement that the entire parabolic tree is a real-analytic arc, and this implies that $f_{\widetilde{c}}$ and $f_{\widetilde{c}^{\ast}}$ (here, and in the sequel, $z^*$ will stand for the complex conjugate of the complex number $z$) are conformally conjugate, proving that $\widetilde{c}$ lies on the real line (or one of its rotates). In order to demonstrate discontinuity of straightening maps, we need to get rid of the non-renormalizability hypothesis; i.e., we need to prove wiggling of umbilical cords for all non-real hyperbolic components. Evidently, in the general case, we have to adopt a different strategy which will be outlined soon.

We also discuss an alternative (and more conformal) reason for straightening maps to be discontinuous. In the presence of more than one critical points, for instance when two critical points are attracted by a single parabolic cycle, one can associate at least two different conformal conjugacy invariants with the parabolic cycle; namely, the \emph{Fatou vector} of the parabolic basin, and the holomorphic fixed point index of the parabolic cycle. The first invariant is related to critical orbits of the polynomial and is preserved by straightening maps. But the second one is in general not preserved by straightening maps (since a hybrid equivalence does not necessarily preserve the external class of a polynomial-like map). Moreover, the Fatou vector can be quasi-conformally deformed giving rise to an analytic family of q.c.\ equivalent parabolic maps \cite[\S 3]{Sa2}. Heuristically speaking, if straightening maps were continuous, they would preserve the geometry of the parameter space. In fact, we prove that continuity of straightening maps would force the above two conformal invariants to be uniformly related along every parabolic arc, and this is a very strong geometric condition that is almost too good to hold. Although we do not know how to rule this out in general, we do show that the relation between Fatou vector and parabolic fixed point index is not uniform for certain low period parabolic arcs of the Tricorn. This shows that straightening maps between certain Tricorn-like sets fail to be continuous essentially because it fails to preserve the `geometry' of connectedness loci. In fact, our methods suggest that any two ``copies'' of the connectedness locus of biquadratic polynomials in the parameter space of cubic polynomials (compare \cite{IK}) are dynamically distinct; i.e., they are not homeomorphic via straightening maps.

We should mention that the main results of this paper can be seen as polynomial dynamics counterparts of some well-known results in the Kleinian group world. More precisely, non-local connectedness of connectedness loci of polynomials is analogous to non-local connectedness of parameter spaces of Kleinian surface groups \cite{Br}, while the analogue of discontinuity of straightening maps (for polynomial parameter spaces) in the Kleinian group setting is given by discontinuity of the action of the modular group at Bers' boundary of Teichm{\"u}ller spaces \cite{KeTh}. Although the proofs of the corresponding results use different techniques, it is interesting to note that the main underlying reason behind these phenomena is the discrepancy between algebraic and geometric limits.

Let us now elaborate on the organization of the paper. In Section~\ref{global}, we will survey some known results about anti-holomorphic dynamics, and the global combinatorial and topological structure of the multicorns. As mentioned earlier, using the implosion techniques developed in \cite{HS,IM}, one can show that the landing of an umbilical cord at $\widetilde{c}$ implies that a loose parabolic tree of $f_{\widetilde{c}}$ would contain a real-analytic arc connecting two bounded Fatou components. From the existence of a small real-analytic arc connecting two bounded Fatou components, we will show in Section~\ref{local_conjugacy} that the characteristic parabolic germs of $f_{\widetilde{c}}$ and $f_{\widetilde{c}^*}$ are conformally conjugate by a local biholomorphism that preserves the critical orbit tails.  This is a fundamental step in our proof. Since there exists an infinite-dimensional family of conformal conjugacy classes of parabolic germs \cite{Ec1, Vor}; heuristically speaking, it is extremely unlikely that the parabolic germs of two conformally different polynomials would be conformally conjugate. The next step in our proof involves extending the local analytic conjugacy between parabolic germs to larger domains, step by step. In Section~\ref{conjugacy_extension}, we first extend this local conjugacy to the entire characteristic Fatou component, and then continue it to a neighborhood of the closure of the characteristic Fatou component. This gives us a pair of conformally conjugate polynomial-like restrictions, and applying a theorem of \cite{I2}, we conclude that some iterates of $f_{\widetilde{c}}$ and $f_{\widetilde{c}^*}$ are globally conjugate by a finite-to-finite holomorphic correspondence. This means that some iterates of $f_{\widetilde{c}}$ and $f_{\widetilde{c}^*}$ are (globally) polynomially semi-conjugate to a common polynomial. The final step in the proof of Theorem~\ref{umbilical_cord_wiggling} is to conclude that $\widetilde{c}$ is conformally conjugate to a real parameter, by using the theory of decompositions of polynomials with respect to composition, which is due to Ritt \cite{R} and Engstrom \cite{Eng}. In Section~\ref{renormalization}, we will recall some general combinatorial and topological facts about straightening maps. Section~\ref{continuity} deals with a continuity property of straightening maps. In particular, we show that straightening maps induce homeomorphisms between the closures of odd period hyperbolic components in the multicorns. Subsequently, in Section~\ref{Discontinuity_of_The_Straightening_Map}, we will use the wiggling behavior of non-real umbilical cords to give a proof of Theorem~\ref{Straightening_discontinuity}. In Section~\ref{geometric_discontinuity}, we state a conjecture on a stronger (and more geometric) form of discontinuity of straightening maps to the effect that the baby multicorns are dynamically different from each other. We also provide positive evidence supporting the conjecture by demonstrating that the original Tricorn is `dynamically' distinct from the period $3$ baby Tricorns.

It is worth mentioning that the proof of discontinuity of straightening maps for general polynomial families given in \cite{I} also involves proving the existence of analytically conjugate polynomial-like restrictions. The main difference is that, in higher dimensional parameter spaces, continuity of straightening maps allows one to find richer perturbations to obtain analytically conjugate polynomial-like maps. Indeed, one of the main technical steps in \cite{I} is to show (using parabolic implosion techniques) that continuity of straightening maps forces certain hybrid equivalences to preserve the moduli of multipliers of repelling periodic points of certain polynomial-like maps, and this implies that the hybrid equivalence can be promoted to an analytic equivalence. On the other hand, the present proof employs a one-dimensional parabolic perturbation to first obtain an analytic conjugacy between parabolic germs, which is then promoted to an analytic conjugacy between polynomial-like restrictions. However, both the proofs have a common philosophy: to show that continuity of straightening maps would force certain hybrid equivalences to preserve some of the `external conformal information'.

The proof of Theorem~\ref{umbilical_cord_wiggling} suggests that unicritical polynomials or anti-polynomials with a parabolic cycle are determined by the conformal conjugacy class of their parabolic germs. This is studied in a subsequent work \cite{IM5}. In that paper, we also apply the techniques used to prove Theorem~\ref{Straightening_discontinuity} and Theorem~\ref{umbilical_cord_wiggling} to prove discontinuity of straightening for the Tricorn-like sets in the parameter space of real cubic polynomials.

We should remark that more generally, one expects the existence of multicorn-like sets in any family of polynomials or rational maps with (at least) two critical orbits such that a pair of critical orbits are symmetric with respect to an anti-holomorphic involution. Evidences of this fact can be found in the recent works on the parameter spaces of certain families of rational maps, such as the family of antipode preserving cubic rationals \cite{BBM1,BBM2}, Blaschke products \cite{CFG}, etc. Although not all of our techniques can be applied to such families of rational maps, the parabolic perturbation arguments, and the local consequence of umbilical cord landing (Section~\ref{local_conjugacy}) do work in a general setting, and paves the way for studying analogous questions for rational maps.

Another recent work where anti-holomorphic parameter spaces and straightening maps play a crucial role is \cite{LLMM2}. In that paper, the authors study a new family of anti-holomorphic dynamical systems given by Schwarz reflection maps. The discontinuity phenomenon demonstrated in the current paper plays an important role in studying the parameter spaces of Schwarz reflection maps.
\bigskip

\noindent\textbf{Acknowledgements.} We would like to thank Arnaud Ch{\'e}ritat, Adam Epstein, John Hubbard, John Milnor, Carsten Lunde Petersen, Dierk Schleicher, and Mitsuhiro Shishikura for many helpful discussions. The first author would like to express his gratitude for the support of JSPS KAKENHI Grant Number 26400115. The second author gratefully acknowledges the support of Deutsche Forschungsgemeinschaft DFG, the Institute for Mathematical Sciences at Stony Brook University, and an endowment from Infosys Foundation during parts of the work on this project.

\section{Anti-holomorphic Dynamics, and Global Structure of The Multicorns}\label{global}
In this section, we briefly recall some known results on anti-holomorphic dynamics, and their parameter spaces, which we will have need for in the rest of the paper.

\subsection{Basic Definitions}
Any unicritical anti-holomorphic polynomial, after an affine change of coordinates, can be written in the form $f_c(z) = \bar{z}^d + c$ for some $d \geq 2$, and $c \in \mathbb{C}$. In analogy to the holomorphic case, the set of all points which remain bounded under all iterations of $f_c$ is called the \emph{Filled Julia set} $K(f_c)$. The boundary of the Filled Julia set is defined to be the \emph{Julia set} $J(f_c)$, and the complement of the Julia set is defined to be its \emph{Fatou set} $F(f_c)$. This leads, as in the holomorphic case, to the notion of \emph{connectedness locus} of degree $d$ unicritical anti-holomorphic polynomials:

\begin{definition}
The \emph{multicorn} of degree $d$ is defined as $\mathcal{M}^{\ast}_d = \{ c \in \mathbb{C} : K(f_c)$ is connected$\}$. The multicorn of degree $2$ is called the \emph{Tricorn}.
\end{definition} 

The basin of infinity and the corresponding B{\"o}ttcher coordinate play a vital role in the dynamics of polynomials. In the anti-holomorphic setting, we need a parallel notion of B{\"o}ttcher coordinates. By \cite[Lemma~1]{Na1}, there is a conformal map $\phi_c$ near $\infty$ that conjugates $f_c$ to $\bar{z}^d$. As in the holomorphic case, $\phi_c$ extends conformally to an equipotential containing $c$, when $c\notin \mathcal{M}_d^*$, and extends as a biholomorphism from $\widehat{\mathbb{C}} \setminus K(f_c)$ onto $\widehat{\mathbb{C}} \setminus \overline{\mathbb{D}}$ when $c \in \mathcal{M}_d^*$.

\begin{definition}[Dynamical Ray]
The dynamical ray $R_c(\theta)$ of $f_c$ at an angle $\theta$ is defined as the pre-image of the radial line at angle $\theta$ under $\phi_c$.
\end{definition}
The dynamical ray $R_c(\theta)$ maps to the dynamical ray $R_c(-d\theta)$ under $f_c$. We refer the readers to \cite[\S 3]{NS}, \cite{Sa} for details on the combinatorics of the landing pattern of dynamical rays for unicritical anti-holomorphic polynomials. The next result was proved by Nakane \cite{Na1}.

\begin{theorem}[Real-Analytic Uniformization]\label{RealAnalUniformization}
The map $\Phi : \mathbb{C} \setminus \mathcal{M}^{\ast}_d \rightarrow \mathbb{C} \setminus \overline{\mathbb{D}}$, defined by $c \mapsto \phi_c(c)$ (where $\phi_c$ is the B\"{o}ttcher coordinate near $\infty$ for $f_c$) is a real-analytic diffeomorphism. In particular, the multicorns are connected.
\end{theorem}

The previous theorem also allows us to define parameter rays of the multicorns. 
\begin{definition}[Parameter Ray]
The parameter ray at angle $\theta$ of the multicorn $\mathcal{M}^{\ast}_d$, denoted by $\mathcal{R}_{\theta}^d$, is defined as $\{ \Phi^{-1}(r e^{2 \pi i \theta}) : r > 1 \}$, where $\Phi$ is the real-analytic diffeomorphism from the exterior of $\mathcal{M}_d^*$ to the exterior of the closed unit disc in the complex plane constructed in Theorem~\ref{RealAnalUniformization}.
\end{definition}

\begin{remark}
Some comments should be made on the definition of the parameter rays. Observe that unlike the multibrot sets, the parameter rays of the multicorns are not defined in terms of the Riemann map of the exterior. In fact, the Riemann map of the exterior of $\mathcal{M}_d^*$ has no obvious dynamical meaning. We have defined the parameter rays via a dynamically defined diffeomorphism of the exterior of $\mathcal{M}_d^*$, and it is easy to check that this definition of parameter rays agrees with the notion of stretching rays (which are dynamically defined objects) in the family of  polynomials $(z^d+a)^d+b$. 
\end{remark}

Let $\omega = \exp(\frac{2\pi i}{d+1})$. The anti-holomorphic polynomials $f_c$ and $f_{\omega c}$ are conformally conjugate via the linear map $z \mapsto  \omega z$. It follows that:

\begin{lemma}[Symmetry]\label{symmetry}
Let $\omega = \exp(\frac{2\pi i}{d+1})$. Then, $f_c \sim f_{\omega c} \sim f_{\omega^2 c} \sim \cdots \sim f_{\omega^d c}$. In particular, $\mathcal{M}_d^{\ast}$ has a $(d+1)$-fold rotational symmetry.
\end{lemma} 
\subsection{Hyperbolic Components of Odd Periods, and Bifurcations}
One of the main features of anti-holomorphic parameter spaces is the existence of abundant parabolics. In particular, the boundaries of odd period hyperbolic components of the multicorns consist only of parabolic parameters.
\begin{lemma}[Indifferent Dynamics of Odd Period]\label{LemOddIndiffDyn}  
The boundary of a hyperbolic component of odd period $k$ consists 
entirely of parameters having a parabolic orbit of exact period $k$. In 
local conformal coordinates, the $2k$-th iterate of such a map has the form 
$z\mapsto z+z^{q+1}+\ldots$ with $q\in\{1,2\}$. 
\end{lemma}

\begin{proof} 
See \cite[Lemma~2.5]{MNS}.
\end{proof}
This leads to the following classification of odd periodic parabolic points.
\begin{definition}[Parabolic Cusps]\label{DefCusp}
A parameter $c$ will be called a {\em cusp point} if it has a parabolic 
periodic point of odd period such that $q=2$ in the previous lemma. Otherwise, it is called a \emph{simple} parabolic parameter.
\end{definition}

In holomorphic dynamics, the local dynamics in attracting petals of parabolic periodic points is well-understood: there is a local coordinate $\psi$ which conjugates the first-return dynamics to translation by $+1$ in a right half plane \cite[\S 10]{M1new}. Such a coordinate $\psi$ is called a \emph{Fatou coordinate}. Thus the quotient of the petal by the dynamics is isomorphic to a bi-infinite cylinder, called the \emph{Ecalle cylinder}. Note that Fatou coordinates are uniquely determined up to addition by a complex constant. 

In anti-holomorphic dynamics, the situation is at the same time restricted and richer. Indifferent dynamics of odd period is always parabolic because for an indifferent periodic point of odd period $k$, the $2k$-th iterate is holomorphic with positive real multiplier, hence parabolic as described above. On the other hand, additional structure is given by the anti-holomorphic intermediate iterate. 

\begin{lemma}[Fatou Coordinates] \cite[Lemma~2.3]{HS}\label{normalization of fatou}
Suppose $z_0$ is a parabolic periodic point of odd period $k$ of $f_c$ with only one petal (i.e., $c$ is not a cusp), and $U$ is a periodic Fatou component with $z_0 \in \partial U$. Then there is an open subset $V \subset U$ with $z_0 \in \partial V$, and $f_c^{\circ k} (V) \subset V$ so that for every $z \in U$, there is an $n \in \mathbb{N}$ with $f_c^{\circ nk}(z)\in 
V$. Moreover, there is a univalent map $\psi \colon V \to \mathbb{C}$ with $\psi(f_c^{\circ k}(z)) = \overline{\psi(z)}+1/2$, and $\psi(V)$ contains a right half plane. This map $\psi$ is unique up to horizontal translation. 
\end{lemma}

The map $\psi$ will be called an  \emph{anti-holomorphic Fatou coordinate} for the petal $V$. The anti-holomorphic iterate interchanges both ends of the Ecalle cylinder, so it must fix one horizontal line around this cylinder (the \emph{equator}). The change of coordinate has been so chosen that the equator is the projection of the real axis.  We will call the vertical Fatou coordinate the \emph{Ecalle height}. Its origin is the equator. Of course, the same can be done in the repelling petal as well. We will refer to the equator in the attracting (respectively repelling) petal as the attracting (respectively repelling) equator.  The existence of this distinguished real line, or equivalently an intrinsic meaning to Ecalle height, is specific to anti-holomorphic maps. 

The Ecalle height of the critical value plays a special role in anti-holomorphic dynamics. The next theorem, which was proved in \cite[Theorem~3.2]{MNS}, shows the existence of real-analytic arcs of non-cusp parabolic parameters on the boundaries of odd period hyperbolic components of the multicorns.

\begin{theorem}[Parabolic Arcs]\label{parabolic arcs}
Let $\widetilde{c}$ be a parameter such that $f_{\widetilde{c}}$ has a parabolic orbit of odd period, and suppose that $\widetilde{c}$ is not a cusp. Then $\widetilde{c}$ is on a parabolic arc in the  following sense: there  exists a real-analytic arc of non-cusp parabolic parameters $c(t)$ (for $t\in\mathbb{R}$) with quasiconformally equivalent but conformally distinct dynamics of which $\widetilde{c}$ is an interior point, and the Ecalle height of the critical value of $f_{c(t)}$ is $t$. 
\end{theorem}

It is known that there exist bifurcations between hyperbolic components of the multicorns across sub-arcs of parabolic arcs. The precise statement is given in the following results, which were proved in \cite[Proposition~3.7, Theorem~3.8, Corollary~3.9]{HS}. The proof of this fact uses the concept of holomorphic fixed point index \cite[\S 12]{M1new}. The main idea is that when several simple fixed points merge into one parabolic point, each of their indices tends to $\infty$, but the sum of the indices tends to the index of the resulting parabolic fixed point, which is finite.

\begin{lemma}[Fixed Point Index on Parabolic Arc]\label{index goes to infinity}
Along any parabolic arc of odd period, the fixed point index is a real valued real-analytic function that tends to $+\infty$ at both ends.
\end{lemma}

\begin{theorem}[Bifurcations Along Arcs]\label{ThmBifArc}
Every parabolic arc of period $k$ intersects the boundary of a hyperbolic component of period $2k$ at the set of points where the fixed-point index is at least $1$, except possibly at (necessarily isolated) points where the index has an isolated local maximum with value $1$. In particular, every parabolic arc has, at both ends, an interval of positive length at which bifurcation from a hyperbolic component of odd period $k$ to a hyperbolic component of period $2k$ occurs.
\end{theorem}

Now let $H$ be a hyperbolic component of odd period $k$, and $\mathcal{C}$ be a parabolic arc on the boundary of $H$. The previous theorem tells that there is an even period hyperbolic component $H'$ (of period $2k$) bifurcating from $H$ across $\mathcal{C}$. Furthermore, the corresponding bifurcation locus $\mathcal{C}\cap \partial H\cap \partial H'$ (from $H$ to $H'$ across $\mathcal{C}$) is precisely the the set of parameters on $\mathcal{C}$ where the fixed-point index is at least $1$, except possibly the (necessarily isolated) points where the index has an isolated local maximum with value $1$. Hence, $\mathcal{C}\cap \partial H\cap \partial H'$ is the union of a sub-arc of $\mathcal{C}$ and possibly finitely many isolated points on $\mathcal{C}$. In our next lemma, we will slightly sharpen the statement of Theorem~\ref{ThmBifArc} by ruling out any such ``accidental intersection'' of $\partial H$ and $\partial H'$. More precisely, we will show that $\mathcal{C}\cap\partial H\cap \partial H'$ contains no isolated point; i.e., it is a sub-arc of $\mathcal{C}$.

Let $c:\mathbb{R}\to\mathcal{C}$ be the critical Ecalle height parametrization of $\mathcal{C}$ (see Theorem~\ref{parabolic arcs}). By \cite[Corollary~5.2]{IM}, there is no bifurcation across the Ecalle height $0$ parameter $c(0)$; i.e., $c(0)\notin \partial H'$. Therefore, $\mathcal{C}\cap \partial H\cap \partial H'$ is contained either in $c(0,+\infty)$ or in $c(-\infty,0)$. We can assume without loss of generality that $\mathcal{C}\cap \partial H\cap \partial H'\subset c(0,+\infty)$. We define $$h_0:= \displaystyle \inf\{h>0:c(h,+\infty)\subset\partial H\cap \partial H'\},\quad \widetilde{h}_0:= \displaystyle \inf\{h>0:c(h)\in\partial H\cap \partial H'\}.$$

Clearly, $0<\widetilde{h}_0\leq h_0$. Observe that if $\mathcal{C}\cap\partial H\cap \partial H'$ contains an isolated point, then $h_0$ would be strictly greater than $\widetilde{h}_0$.

\begin{lemma}[No Accidental Bifurcation]\label{No_Accidental_Bifurcation}
$h_0=\widetilde{h}_0$. Consequently, $\mathcal{C}\cap\partial H\cap \partial H'$ contains no isolated point.
\end{lemma}
 
\begin{figure}[ht!]
\includegraphics[scale=0.2]{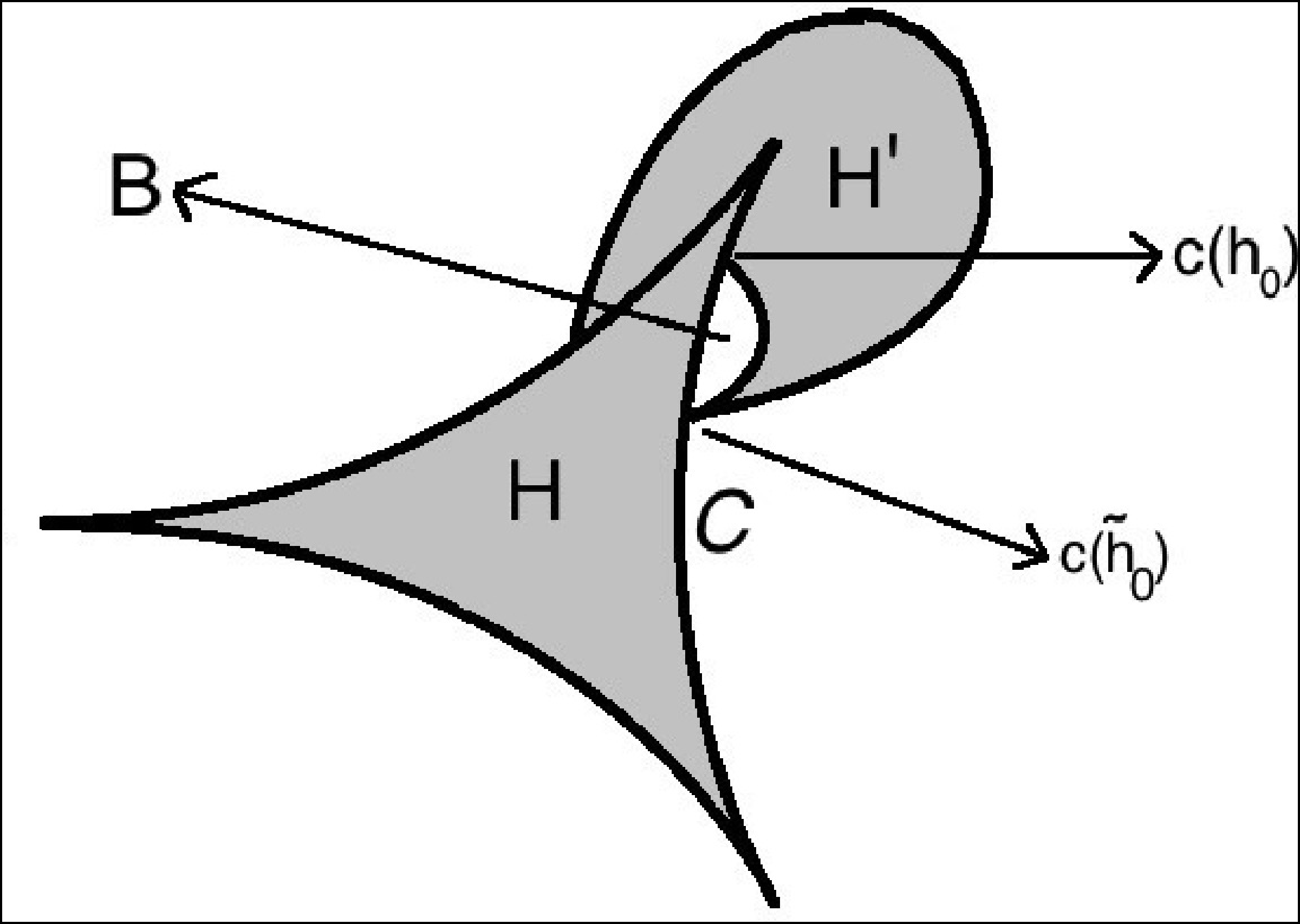}

\caption{The hyperbolic components $H$ and $H'$ are shown in grey. If $\partial H\cap\partial H'$ contained an isolated point, then $\overline{H}\cup\overline{H'}$ would have a bounded complementary component $B$. Therefore, all parameters on $(\overline{B}\cap\partial H')\setminus\mathcal{C}$ will be irrationally indifferent.}
\label{no_isolated}
\end{figure}
\begin{proof}
Let us assume that $h_0>\widetilde{h}_0$. So $\mathbb{C}\setminus(\overline{H}\cup\overline{H'})$ has a bounded component $B$. Note that $\overline{B}\setminus\{c(\widetilde{h}_0)\}$ is contained in the interior of $\mathcal{M}_d^*$ (compare Figure~\ref{no_isolated}), and hence cannot contain a parabolic parameter of even period since every parabolic parameter of even period is the landing point of some external parameter ray of $\mathcal{M}_d^*$ (compare \cite[Lemma~7.4]{MNS}). Therefore every point of $(\overline{B}\cap\partial H')\setminus\mathcal{C}$ must have an irrationally indifferent $2k$-periodic cycle, and their multipliers depend continuously on the parameter. Hence, this multiplier map must be constant on $(\overline{B}\cap\partial H')\setminus\mathcal{C}$. This means that there is a $\theta$ in $(\mathbb{R}\setminus\mathbb{Q})/\mathbb{Z}$ such that each parameter on $(\overline{B}\cap\partial H')\setminus\mathcal{C}$ has a $2k$-periodic cycle of multiplier $e^{2\pi i\theta}$. 

We claim that the set of parameters $c$ in $\mathcal{M}_d^*$ such that $f_c$ has a neutral periodic orbit of given period $2k$ and given multiplier $\mu$ (where $\vert \mu\vert= 1$) is finite. Once this claim is established, we would arrive at a contradiction with the conclusion of the previous paragraph, and this would imply that $h_0=\widetilde{h}_0$.

We proceed to the proof of the claim. When $\mu$ is a root of unity, the desired finiteness follows from \cite[Lemma~2.7]{MNS}. Let us now assume that $\mu=e^{2\pi i\theta}$ for some $\theta\in (\R\setminus\Q)/\Z$. We embed the family $\{f_c\}$ in a two-parameter family of polynomials 
$$
\{P_{a,b}(z) = (z^d+a)^d+b;\,\,(a,b) \in \mathbb{C}^2\}
$$ 
which depends analytically on the parameters $a$ and $b$. The critical points of these polynomials are $0$ and the $d$ roots of the equation $z^d+a = 0$ such that each critical point has multiplicity $d-1$; however, these maps have only two critical values $a^d+b$ and $b$. Since $f_c^{\circ 2} = P_{\bar{c},c}$, our family can be regarded as a real 
$2$-dimensional slice (namely, $a = \bar{b}$) of the family $\{P_{a,b}\}$. A $2k$-cycle $\{z_j = f_c^{\circ j}(z_0);0 \leq j\leq 2k-1\}$ of $f_c$ gives two distinct $k$-cycles $(z_{2i})$ and $(z_{2i+1})$ (for $i=0,\cdots,k-1$) of $f_c^{\circ 2}$; the multipliers of these two cycles are equal to $\mu$ and $\overline{\mu}$. It now suffices to show that the set 
$$
\Lambda(k,\mu)
:= \lbrace (a,b) \in \mathbb{C}^2 :  P_{a,b}\ \textrm{has two distinct orbits of exact period}\ k
$$
$$
\textrm{with multipliers}\ \mu\ \textrm{and}\ \overline{\mu} \rbrace
$$
is finite. To this end, consider the complex numbers $a, b, z_0, z_1$ satisfying the four algebraic equations 
\begin{eqnarray} 
P_{a,b}^{\circ k}(z_0) - z_0 &=& 0, \quad
(P_{a,b}^{\circ k})^{\prime}(z_0) = \mu,
\label{EqBezout1}
\\ 
P_{a,b}^{\circ k}(z_1) - z_1 &=& 0, \quad 
(P_{a,b}^{\circ k})^{\prime}(z_1) = \overline{\mu}. 
\label{EqBezout2}
\end{eqnarray}
Let $R_{\mu}(a,b)$ (respectively $R_{\overline{\mu}}(a,b)$) be the resultant of the two polynomials $P_{a,b}^{\circ k}(z)-z$ and $(P_{a,b}^{\circ k})^{\prime}(z)-\mu$ (respectively, of $P_{a,b}^{\circ k}(z)-z$ and $(P_{a,b}^{\circ k})^{\prime}(z)-\overline{\mu}$). Then a solution $z_0$ of \eqref{EqBezout1} (respectively, a solution $z_1$ of \eqref{EqBezout2}) exists if and only if $R_{\mu}(a,b) = 0$ (respectively, $R_{\overline\mu}(a,b) = 0$). It follows that 
$$
\Lambda(k,\mu) \subset \{R_{\mu}(a,b) = R_{\overline{\mu}}(a,b) = 0\}.
$$
Since $\mu \neq \overline{\mu}$, the two algebraic varieties $R_{\mu}(a,b) = 0$ and $R_{\overline\mu}(a,b) = 0$ are distinct. By B\'{e}zout's theorem, the intersection of the two algebraic varieties $R_{\mu}(a,b) = 0$ and $R_{\overline\mu}(a,b) = 0$ is either a finite set or a common irreducible component with unbounded projection over each variable. Suppose that they have a common irreducible component, say $S$. For each $(a,b)\in S$, the map $P_{a,b}$ has two distinct irrationally neutral cycles, and by \cite[Theorem~1.1]{BCLOS}, both critical orbits of $P_{a,b}$ are non-escaping (more precisely, there is a recurrent critical orbit associated to each of the two irrationally neutral cycles). Hence, $S$ is contained in the connectedness locus of the family of monic centered polynomials of degree $d^2$, which is compact by \cite{BH}. But this contradicts unboundedness of $S$. Therefore, the intersection of the above two algebraic curves is finite; and hence, $\Lambda(k,\mu)$ is a finite set. This completes the proof of the claim and the lemma.
\end{proof}

Following \cite{MNS}, we classify parabolic arcs into two types.

\begin{definition}[Root Arcs and Co-Root Arcs]\label{DefRootArc} 
We call a parabolic arc a \emph{root arc} if, in the dynamics of any 
parameter on this arc, the parabolic orbit disconnects the Julia set. 
Otherwise, we call it a \emph{co-root arc}.
\end{definition}

\begin{definition}[Characteristic Parabolic Point]
Let $f_c$ have a parabolic periodic point. The unique Fatou component of $f_c$ containing the critical value $c$ is called the \emph{characteristic Fatou component}. The \emph{characteristic parabolic point} of $f_c$ is the unique parabolic point on the boundary of the characteristic Fatou component.
\end{definition}

\begin{definition}[Rational Lamination]
The rational lamination of a holomorphic or anti-holomorphic polynomial $f$ (with connected Julia set) is defined as an equivalence relation on $\mathbb{Q}/\mathbb{Z}$ such that $\theta_1 \sim \theta_2$ if and only if the dynamical rays $R(\theta_1)$ and $R(\theta_2)$ land at the same point of $J(f)$. The rational lamination of $f$ is denoted by $\lambda(f)$.
\end{definition}

The structure of the hyperbolic components of odd period plays an important role in the global topology of the parameter spaces. Let $H$ be a hyperbolic component of odd period $k\neq1$ (with center $c_0$) of the multicorn $\mathcal{M}_d^{\ast}$. The first return map of the closure of the characteristic Fatou component of $c_0$ fixes exactly $d+1$ points on its boundary. Only one of these fixed points disconnects the Julia set, and is the landing point of two distinct dynamical rays at $2k$-periodic angles. Let the set of the angles of these two rays be $S' = \{\alpha_1,\alpha_2 \}$. Each of the remaining $d$ fixed points is the landing point of precisely one dynamical ray at a $k$-periodic angle; let the collection of the angles of these rays be $S = \{ \theta_1, \theta_2, \cdots, \theta_d \}$. We can, possibly after renumbering, assume that $0 < \alpha_1 < \theta_1 < \theta_2 < \cdots < \theta_d < \alpha_2$ and $\alpha_2 - \alpha_1 < \frac{1}{d}$. 

By \cite[Theorem~1.2]{MNS}, $\partial H$ is a simple closed curve consisting of $d+1$ parabolic arcs, and the same number of cusp points such that every arc has two cusp points at its ends. Exactly one of these $d+1$ parabolic arcs is a root arc, and the parameter rays at angles $\alpha_1$ and $\alpha_2$ accumulate on this arc. The characteristic parabolic point in the dynamical plane of any parameter on this root arc is the landing point of precisely two dynamical rays at angles $\alpha_1$ and $\alpha_2$. The rest of the $d$ parabolic arcs on $\partial H$ are co-root arcs. Each of these co-root arcs contains the accumulation set of exactly one parameter ray at an angle $\theta_i$, and such that the characteristic parabolic point in the dynamical plane of any parameter on this co-root arc is the landing point of precisely one dynamical ray at angle $\theta_i$. Furthermore, the rational lamination remains constant throughout the closure of the hyperbolic component $H$ except at the cusp points.

By \cite[Theorem~5.6]{NS}, every even period hyperbolic component $H'$ of $\mathcal{M}_d^*$ is homeomorphic to $\mathbb{D}$, and the corresponding multiplier map $\mu_{H'} : H' \to \mathbb{D}$ is a real-analytic $(d-1)$-fold branched cover ramified only over the origin.

\begin{definition}[Internal Rays of Even Period Components]
An \emph{internal ray} of an even period hyperbolic component~$H'$ of $\mathcal{M}_d^*$ is an
arc~$\gamma \subset H'$ starting at the center such that
there is an angle~$\theta$ with~$\displaystyle \mu_{H'}(\gamma) = \lbrace re^{2\pi i\theta}: r \in [0,1)\rbrace$.
\end{definition}

\begin{remark}
Since~$\mu_{H'}$ is a~$(d-1)$-to-one map, an internal ray of~$H'$
with a given angle is not uniquely defined. In fact, a hyperbolic component has~$(d-1)$ internal rays with
any given angle~$\theta$.
\end{remark}

Let us record the following basic property of internal rays of even period hyperbolic components. Although a proof of this fact has never appeared before, we believe that the result is somewhat folklore.

\begin{lemma}[Landing of Internal Rays for Even Period Components]\label{even_rays_land}
For every hyperbolic component $H'$ of even period, all internal rays land. The landing point of an internal parameter ray at angle $\theta$ has an indifferent periodic orbit with multiplier $e^{2\pi i\theta}$. If the period of this orbit is odd, then $\theta=0$, and the landing point is a parabolic cusp.
\end{lemma}

\begin{proof}
Let $2k$ be the period of the hyperbolic component $H'$. Every parameter $c$ in the accumulation set of an internal ray at angle $\theta$ has an indifferent periodic point $z$ satisfying $f_c^{\circ 2k}(z) =z$, $(f_c^{\circ 2k})'(z)=e^{2\pi i\theta}$. So an internal ray lands whenever such boundary parameters are isolated. If $H'$ does not bifurcate from an odd period hyperbolic component, then indifferent parameters of a given multiplier  are isolated on $\partial H'$ (recall that the set of parameters $c$ in $\mathcal{M}_d^*$ such that $f_c$ has a \emph{non-repelling} periodic orbit of given period $2k$ and given multiplier $\mu$, where $\vert \mu\vert\leq 1$, is finite). On the other hand, if $H'$ bifurcates from an odd period hyperbolic component, then indifferent parameters of multiplier $e^{2\pi i\theta}$ may be non-isolated on $\partial H'$ only if $\theta=0$, and the parabolic orbit of the corresponding parameters have \emph{odd period} $k$. Therefore we only need to consider this one exceptional case. In all other cases, the candidate accumulation set of an internal ray is discrete, and hence the ray must land.

Let $\mathcal{R}$ be an internal ray at angle $0$ of $H'$ (bifurcating from an odd period component $H$ of period $k$), and $c$ be an accumulation point of $\mathcal{R}$ such that $f_c$ has a $k$-periodic parabolic cycle of multiplier $1$. We claim that $c$ must be a cusp point of period $k$ on $\partial H$. Since cusp points of a given period are isolated \cite[Lemma~2.9]{MNS}, it would follow that $\mathcal{R}$ lands at a cusp point, completing the proof of the lemma. To prove the claim, let us choose a sequence $\lbrace c_n\rbrace \subset \mathcal{R}$ with $c_n\to c$ such that each of the two $k$-periodic attracting cycles of $f_{c_n}^{\circ 2}$ has multiplier $\lambda_n$ (in general, these two attracting cycles have complex conjugate multipliers, but since $\mathcal{R}$ is an internal ray at angle $0$, the multipliers are real, and hence equal). If $c$ has a simple parabolic cycle, then the residue fixed point index of the parabolic cycle of $f_c^{\circ 2}$ is equal to 
\begin{equation*}
\displaystyle \lim_{n\to \infty} \frac{1}{1-\lambda_n}+\frac{1}{1-\lambda_n}=\displaystyle \lim_{n\to \infty} \frac{2}{1-\lambda_n}=+\infty,
\end{equation*}
which is impossible. Therefore, $c$ must be a cusp point of period $k$ on $\partial H$.
\end{proof}

\begin{remark}
It follows from the above lemma that if $H'$ is an even period hyperbolic component (of period $2k$) bifurcating from an odd period component $H$ (of period $k$), then no internal ray of $H'$ accumulates on the parabolic arcs of $H$.
\end{remark}

The above lemma has an interesting corollary. We will use the terminologies of Lemma~\ref{No_Accidental_Bifurcation}. For any $h$ in $\mathbb{R}$, let us denote the residue fixed point index of the unique parabolic cycle of $f_{c(h)}^{\circ 2}$ by $\ind_{\mathcal{C}}(f_{c(h)}^{\circ 2})$. By Lemma~\ref{No_Accidental_Bifurcation}, we can assume without loss of generality that the set of parameters on $\mathcal{C}$ across which bifurcation from $H$ to $H'$ occurs is precisely $c[h_0,+\infty)$; i.e., $\mathcal{C}\cap \partial H\cap \partial H'= c[h_0,+\infty)$.

\begin{corollary}\label{index_increasing}
The function $$\ind_{\mathcal{C}}:\mathbb{R} \to \mathbb{R},\ h \mapsto \ind_{\mathcal{C}}(f_{c(h)}^{\circ 2})$$
is strictly increasing for $h\geq h_0$.
 \end{corollary}
  
 \begin{figure}[ht!]
\begin{minipage}{0.4\linewidth}
\begin{flushleft}
\includegraphics[scale=0.27]{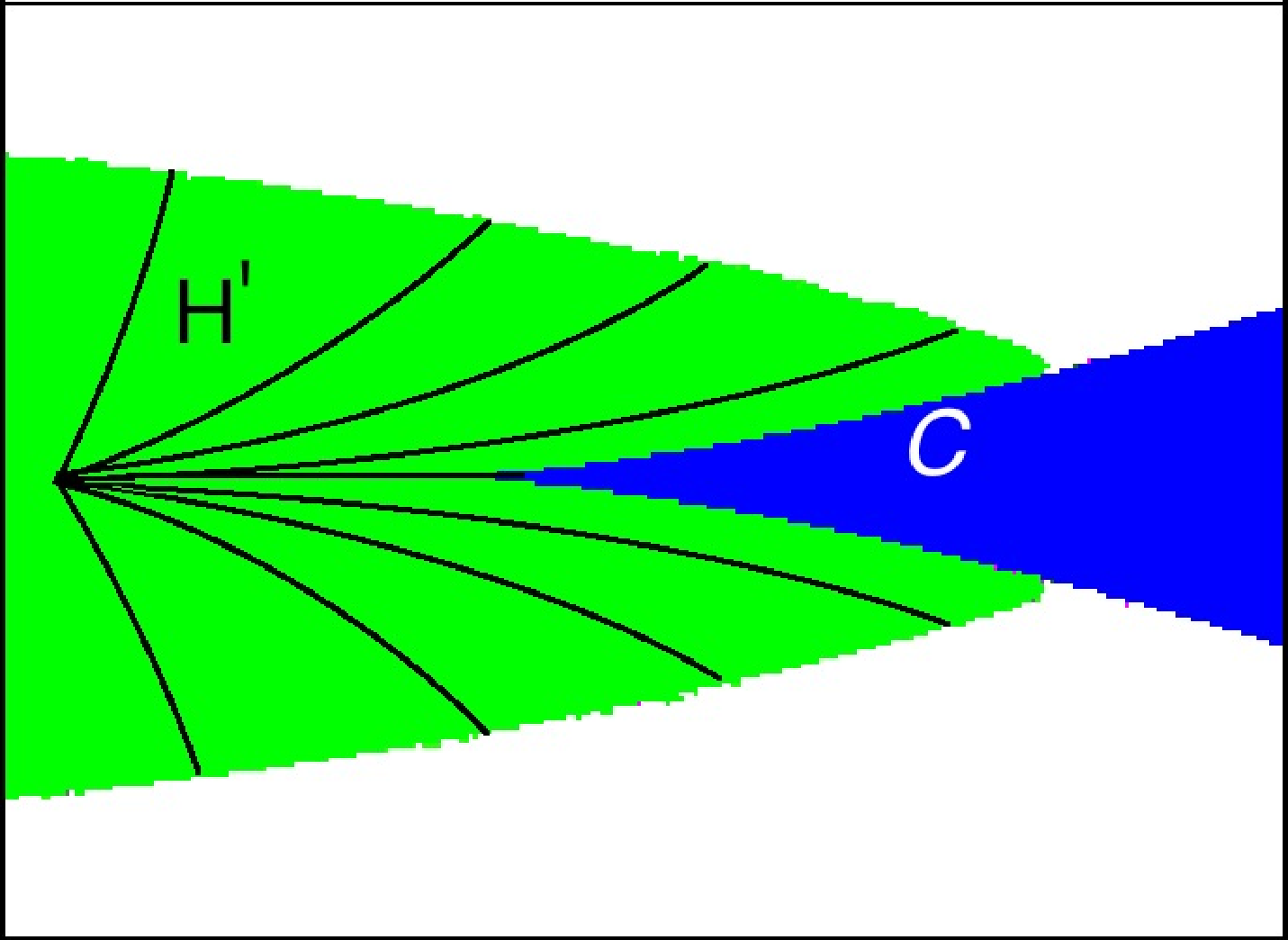}
\end{flushleft}
\end{minipage}  
\begin{minipage}{0.56\linewidth}
\begin{flushright}
\includegraphics[scale=0.55]{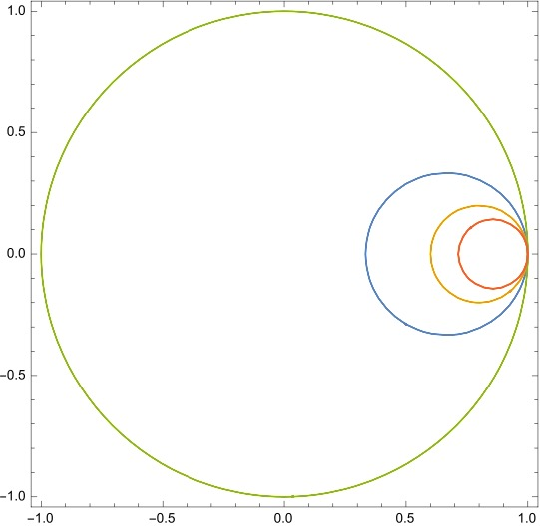}
\end{flushright}
\end{minipage}
\caption{Left: A schematic picture of internal rays of an even period hyperbolic component bifurcating from an odd period hyperbolic component. Right: The circles $C_h$ (in the uniformizing plane) centered at $(1-\frac{1}{\tau_h})$ with radius $\frac{1}{\tau_h}$ for various values of $\tau_h$. All of these circles touch at $1$. As $\tau_h$ increases, the circles get nested. As $c_n\to c(h)$, $\mu_{H'}(c_n)$ converges to $1$ asymptotically to the circle $C_h$.}
\label{convergence_pattern}
\end{figure}
 
 \begin{proof}
 Pick $h\geq h_0$. It follows from the proof of Lemma~\ref{even_rays_land} that an internal ray $\mathcal{R}$ at angle $0$ of $H'$ lands at the cusp point $\displaystyle\lim_{h\to+\infty} c(h)$, and the impression of $\mathcal{R}$ contains the sub-arc $c[h_0,+\infty)$ (of the parabolic arc $\mathcal{C}$). So we can choose a sequence $(H'\ni) c_n\to c(h)\in \mathcal{C}$ such that $\mu_{H'}(c_n)=r_n e^{2\pi i\theta_n}=x_n+i y_n,$ where $r_n\uparrow 1$, and $\theta_n \to 0$. It follows that $$\ind_{\mathcal{C}}(f_{c(h)}^{\circ 2})=\displaystyle \lim_{n\to \infty}  \left(\frac{1}{1-r_n e^{2\pi i\theta_n}}+\frac{1}{1-r_n e^{-2\pi i\theta_n}}\right)= \displaystyle \lim_{n\to \infty}  \left(\frac{2(1-x_n)}{(1-x_n)^2+y_n^2} \right).$$

Set $\tau_h:=\ind_{\mathcal{C}}(f_{c(h)}^{\circ 2})$. The above relation implies that as $n$ tends to infinity, $\mu_{H'}(c_n)$ tends to $1$ asymptotically to the circle $C_h :=$
$$\frac{2(1-x)}{(1-x)^2+y^2}=\tau_h,\ \mathrm{or}\ \left(x-\left(1-\frac{1}{\tau_h}\right)\right)^2+y^2=\frac{1}{\tau_h^2}.$$

For any fixed $\theta \neq 0$ and sufficiently close to $0$, parameters (on $\mathcal{C}$) with larger critical Ecalle height $h$ are approximated by parameters $c$ (in $H'$) with smaller values of $r:=\vert \mu_{H'}(c)\vert$ (compare Figure~\ref{convergence_pattern}(Left)). But it is easy to see either by direct computation or from Figure~\ref{convergence_pattern}(Right) that a fixed radial line at angle $\theta \neq 0$ intersects circles $C_h$ (here, we consider the point of intersection that is closer to the unit circle $\mathbb{S}^1$) with smaller radii (i.e., larger $\tau_h$) at points with smaller values of $r$ (by the nesting pattern of the circles). The upshot of this analysis is that larger values of $h$ correspond to larger (strictly speaking, no smaller) values of $\tau_h$. This implies that the function $\ind_{\mathcal{C}}$ is non-decreasing on $\left[h_0,+\infty\right)$. Since $\ind_{\mathcal{C}}$ is a non-constant real-analytic function, it must be strictly increasing on $\left[h_0,+\infty\right)$.
\end{proof}

It should be mentioned that unlike for the multibrot sets, not every (external) parameter ray of the multicorns land at a single point. 

\begin{theorem}[Non-Landing Parameter Rays]\label{most rays wiggle}
The accumulation set of every parameter ray accumulating on the boundary of a hyperbolic component of \emph{odd} period (except period one) of $\mathcal{M}_d^{\ast}$ contains an arc of positive length.
\end{theorem}

See \cite[Theorem~1.1, Theorem~4.2]{IM} for a detailed account on non-landing parameter rays, and for a complete classification of which rays land, and which ones do not. 

\subsection{Parabolic Tree}
We will need the concept of parabolic trees, which are defined in analogy with Hubbard trees for post-critically finite polynomials. The proofs of the basic properties of the tree can be found in \cite[Lemma~3.5, Lemma~3.6]{S1a}.

\begin{definition}[Parabolic Tree]
If $c$ lies on a parabolic root arc of odd period $k$, we define a \emph{loose parabolic tree} of $f_c$ as a minimal tree within the filled Julia set that connects the parabolic orbit and the critical orbit such that it intersects the closure of any bounded Fatou component at no more than two points.
\end{definition}

Since the filled Julia set of a parabolic polynomial is locally connected, and hence path connected, any loose parabolic tree connecting the parabolic orbit is uniquely defined up to homotopies within bounded Fatou components. It is easy to see that any loose parabolic tree intersects the Julia set in a Cantor set, and these points of intersection are the same for any loose tree (note that for simple parabolics, any two periodic Fatou components have disjoint closures).

By construction, the forward image of a loose parabolic tree is contained in a loose parabolic tree. A simple standard argument (analogous to the post-critically finite case) shows that the boundary of the critical value Fatou component intersects the tree at exactly one point (the characteristic parabolic point), and the boundary of any other bounded Fatou component meets the tree in at most $d$ points, which are iterated pre-images of the characteristic parabolic point \cite[Lemma~3.5]{S1a} \cite[Lemma~3.2, Lemma~3.3]{EMS}. The critical value is an endpoint of any loose parabolic tree. All branch points of a loose parabolic tree are either in bounded Fatou components or repelling (pre-)periodic points; in particular, no parabolic point (of odd period) is a branch point.

\section{Umbilical Cord Landing, and Conformal Conjugacy of Parabolic Germs}\label{local_conjugacy}

\textbf{A standing convention:} In the rest of the paper, we will denote the complex conjugate of a complex number $z$ either by $\overline{z}$ or by $z^*$. The complex conjugation map will be denoted by $\iota$, i.e., $\iota(z)=z^*$. The image of a set $U$ under complex conjugation will be denoted as $\iota(U)$, and the topological closure of $U$ will be denoted by $\overline{U}$.

Our goal in this section is to apply the perturbation techniques developed in \cite[\S 4]{HS}, \cite[\S 2]{IM} to prove a local consequence of umbilical cord landing. We will work with a fixed hyperbolic component $H$ of odd period $k\neq 1$. Let $\mathcal{C}$ be the root arc on $\partial H$, $\widetilde{c}$ be the Ecalle height $0$ parameter on $\mathcal{C}$, and $z_1$ be the characteristic parabolic point of $f_{\widetilde{c}}$. Assume further that the dynamical rays $R_{\widetilde{c}}(\theta)$ and $R_{\widetilde{c}}(\theta')$ land at the characteristic parabolic point $z_1$.  By symmetry, there is a hyperbolic component $\iota(H)$ (which is just the reflection of $H$ with respect to the real line) of the same period $k$ such that $\widetilde{c}^*$ is the Ecalle height $0$ parameter on the root arc $\iota(\mathcal{C})$ of $\iota(H)$. The characteristic parabolic point of $f_{\widetilde{c}^*}$ is $z_1^*$.

We begin with an elementary lemma.  

\begin{lemma}\label{primitive}
Any two bounded Fatou components of $f_{\widetilde{c}}$ have disjoint closures.
\end{lemma}

\begin{proof}
Let $U_1$ and $U_2$ be two distinct Fatou components of $f_{\widetilde{c}}$ with $\overline{U_1}\cap \overline{U_2} \neq \emptyset$. By taking iterated forward images, we can assume that $U_1$ and $U_2$ are periodic. Then the intersection consists of a repelling periodic point $x$ of some period $n$. 

It is easy to see that the first return map of $\overline{U_1}$ (and of $\overline{U_2}$) fixes $x$, and $x$ disconnects the Julia set; hence, $x$ is the `root' of $U_1$ (as well as of $U_2$). It follows from \cite[Corollary~4.2]{NS} that $n=k$. But this contradicts the fact that every periodic Fatou component of $f_{\widetilde{c}}$ has exactly one root, and there is exactly one cycle of periodic bounded Fatou components (for instance, use the fact that on the root parabolic arc on $\partial H$, the attracting periodic points of both $U_1$ and $U_2$ would merge with the root $x$ yielding a double parabolic point!). This contradiction proves the lemma.
\end{proof} 

The following lemma essentially states that landing of an umbilical cord at $\widetilde{c}$ implies a (local) regularity property of a loose parabolic tree of $f_{\widetilde{c}}$. Although the proof can be extracted from \cite[Lemma~5.10, Theorem~6.1]{HS}, we prefer working out the details here as the organization of the present paper differs from that of \cite{HS}.
 
\begin{lemma}\label{straight_equator}
If there is a path $p : [0,\delta] \rightarrow \mathbb{C}$ with $p(0) = \widetilde{c}$, and $p((0,\delta]) \subset \mathcal{M}_d^* \setminus \overline{H}$, then the repelling equator at $z_1$ is contained in a loose parabolic tree of $f_{\widetilde{c}}$.
\end{lemma}

\begin{proof}
Since any two bounded Fatou components of $f_{\widetilde{c}}$ have disjoint closures, and the inverse images of the characteristic parabolic point $z_1$ are dense on the Julia set, it follows that any parabolic tree must traverse infinitely many bounded Fatou components, and intersect their boundaries at pre-parabolic points. Furthermore, any loose parabolic tree intersects the Julia set at a Cantor set of points. We first claim that the part of any loose parabolic tree contained in the repelling petal of $z_1$ intersects the Julia set entirely along the repelling equator. To do this, we will assume the contrary, and will arrive at a contradiction.

If the part of the parabolic tree contained in the repelling petal were not contained in the equator, then there would be a point $w_0$ (say, repelling pre-periodic) in the intersection having Ecalle height $h >0$. We construct a sequence $(w_n)$ so that $w_{n+1} := f_{\widetilde{c}}^{-\circ 2k}(w_n)$ choosing a local branch of $f_{\widetilde{c}}^{-\circ 2k}$ that fixes $z_1$, and so that all $w_n$ are in the repelling petal of $z_1$. Therefore, $w_n \to z_1$ as $n \to \infty$, and all $w_n$ have the same Ecalle height $h$. Similarly, let $w_{n}' := f_{\widetilde{c}}^{\circ k}\left( w_n\right)$, then $w_{n}' \to z_1$, and all these points have Ecalle heights $-h$. As $w_0$ is on the parabolic tree, which is invariant, it follows that $w_0$ is accessible from outside of the filled Julia set on both sides of the tree, so each $w_n$ and $w_{n}'$ is the landing point of (at least) two dynamic rays, `above' and `below' the tree. If $\theta_n$ is the angle of the dynamical ray landing at $w_n$ from below, then it follows that $\theta_n \to \theta$ (say) as $n \to \infty$, and $R_{\widetilde{c}}(\theta_n)$ traverses (at least) the interval $\left[-h/2, h/2\right]$ of (outgoing) Ecalle heights. Analogously, if $\theta_{n}'$ is the angle of the dynamical ray landing at $w_{n}'$ from above, then it follows that $\theta_{n}' \to \theta'$ as $n \to \infty$, and $R_{\widetilde{c}}(\theta_{n}')$ traverses (at least) the interval $\left[-h/2, h/2\right]$ of (outgoing) Ecalle heights (see Figure~\ref{comb_dynamic_parameter}(Left)).

\begin{figure}[ht!]
\begin{minipage}{0.48\linewidth}
\begin{center}
\includegraphics[width=0.99\linewidth]{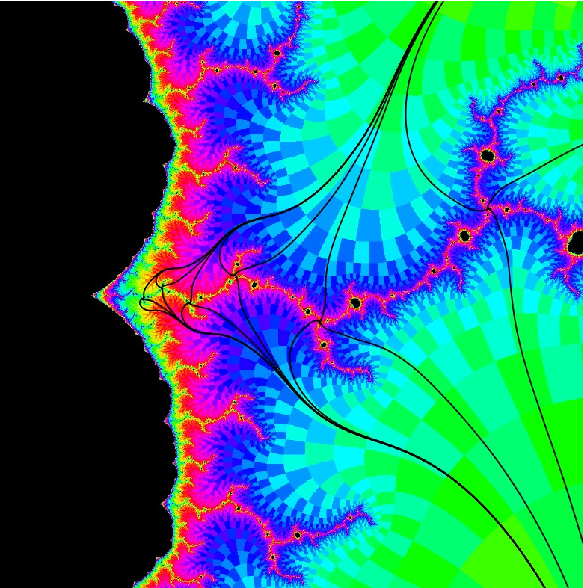}
\end{center}
\end{minipage}
\hspace{1mm}
\begin{minipage}{0.48\linewidth}
\begin{center}
 \includegraphics[width=0.88\linewidth]{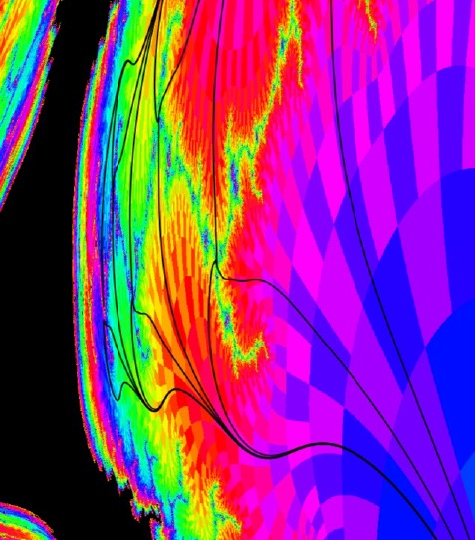}
 \end{center}
\end{minipage}
\caption{Left: Dynamical rays crossing the repelling equator in the dynamical plane. Right: The corresponding parameter rays obstructing the existence of the required path $p$.}
\label{comb_dynamic_parameter}
\end{figure}

The dynamical rays at angles $\theta_n$ and $\theta_{n}'$, and their landing points depend equicontinuously (i.e., the uniform continuous dependence on the parameter is independent of $n$) on the parameter (as they are pre-periodic rays). The projection of these rays onto the outgoing Ecalle cylinder is also continuous. Hence, there exists a neighborhood $U$ of $\widetilde{c}$ in the parameter space such that if $c' \in \overline{U} \setminus H$, then the projection of the dynamical rays $R_{c'}(\theta_n)$ and $R_{c'}(\theta_{n}')$ onto the outgoing cylinder of $f_{c'}^{\circ 2k}$ traverse the interval of (outgoing) Ecalle heights $\left[-h/3, h/3\right]$ .

By assumption, there is a path $p : [0,\delta] \rightarrow \mathbb{C}$ with $p(0) = \widetilde{c}$, and $p((0,\delta]) \subset \mathcal{M}_d^* \setminus \overline{H}$. By choosing a smaller $\delta$, we can assume that $p((0,\delta]) \subset \overline{U} \setminus \overline{H}$. For $s > 0$, the critical orbit of $f_{p(s)}$ ``transits'' from the incoming Ecalle cylinder to the outgoing cylinder; as $s \downarrow 0$, the image of the critical orbit in the outgoing Ecalle cylinder has (outgoing) Ecalle height tending to $0$, while the phase tends to infinity \cite[Lemma~2.5]{IM}. Therefore, there is $s \in \left(0, \delta\right)$ arbitrarily close to $0$ for which the critical value, projected into the incoming cylinder, and sent by the transit map to the outgoing cylinder, lands on the projection of the dynamical rays $R_{p(s)}(\theta_n)$ (or $R_{p(s)}(\theta_{n}')$). But in the dynamics of $f_{p(s)}$, this means that the critical value is in the basin of infinity, i.e., such a parameter $p(s)$ lies outside $\mathcal{M}_d^*$. This contradicts our assumption that $p((0,\delta]) \subset \mathcal{M}_d^* \setminus \overline{H}$, and proves that the part of any loose parabolic tree contained in the repelling petal of $z_1$ must intersect the Julia set entirely along the repelling equator.

In fact, the above argument essentially shows that the existence of any dynamical ray (in the repelling petal at $z_1$) traversing an interval of outgoing Ecalle heights $\left[-x, x\right]$ with $x>0$ would destroy the existence of the required path $p$ (compare Figure~\ref{comb_dynamic_parameter}). In other words, for the existence of such a path $p$, no dynamical ray should `cross' the repelling equator. Therefore, the repelling equator is contained in the filled Julia set $K(f_{\widetilde{c}})$; i.e., the repelling equator forms the part of a loose parabolic tree.
\end{proof}

So far, we have more or less proceeded in the same direction as in \cite{HS}. More precisely, we have showed that in order that an umbilical cord lands, the corresponding anti-holomorphic polynomial must have a loose parabolic tree whose intersection with the repelling petal is an analytic arc (observe that the equator is an analytic arc; i.e., the image of the real line under a biholomorphism). But without any assumption on non-renormalizability, we \emph{cannot} conclude anything about the global structure of the parabolic tree. To circumvent this problem, we will adopt a different `local to global' principle. The following lemma shows that the local regularity of the parabolic tree, established in the previous lemma, has a very surprising consequence on the corresponding parabolic germs.

\begin{lemma}[Local Analytic Conjugacy of Parabolic Germs]\label{path gives analytically conjugate parabolic germs}
If the repelling equator of $f_{\widetilde{c}}$ at $z_1$ is contained in a loose parabolic tree, then the parabolic germs given by the restrictions of $f_{\widetilde{c}}^{\circ 2k}$ and $f_{\widetilde{c}^*}^{\circ 2k}$ at their characteristic parabolic points are conformally conjugate by a local biholomorphism that maps $f_{\widetilde{c}}^{\circ kr}(\widetilde{c})$ to $f_{\widetilde{c}^*}^{\circ kr}(\widetilde{c}^*)$, for $r$ large enough. 
\end{lemma}

\begin{figure}[ht!]
\begin{center}
\begin{minipage}{0.48\linewidth}
\includegraphics[scale=0.16]{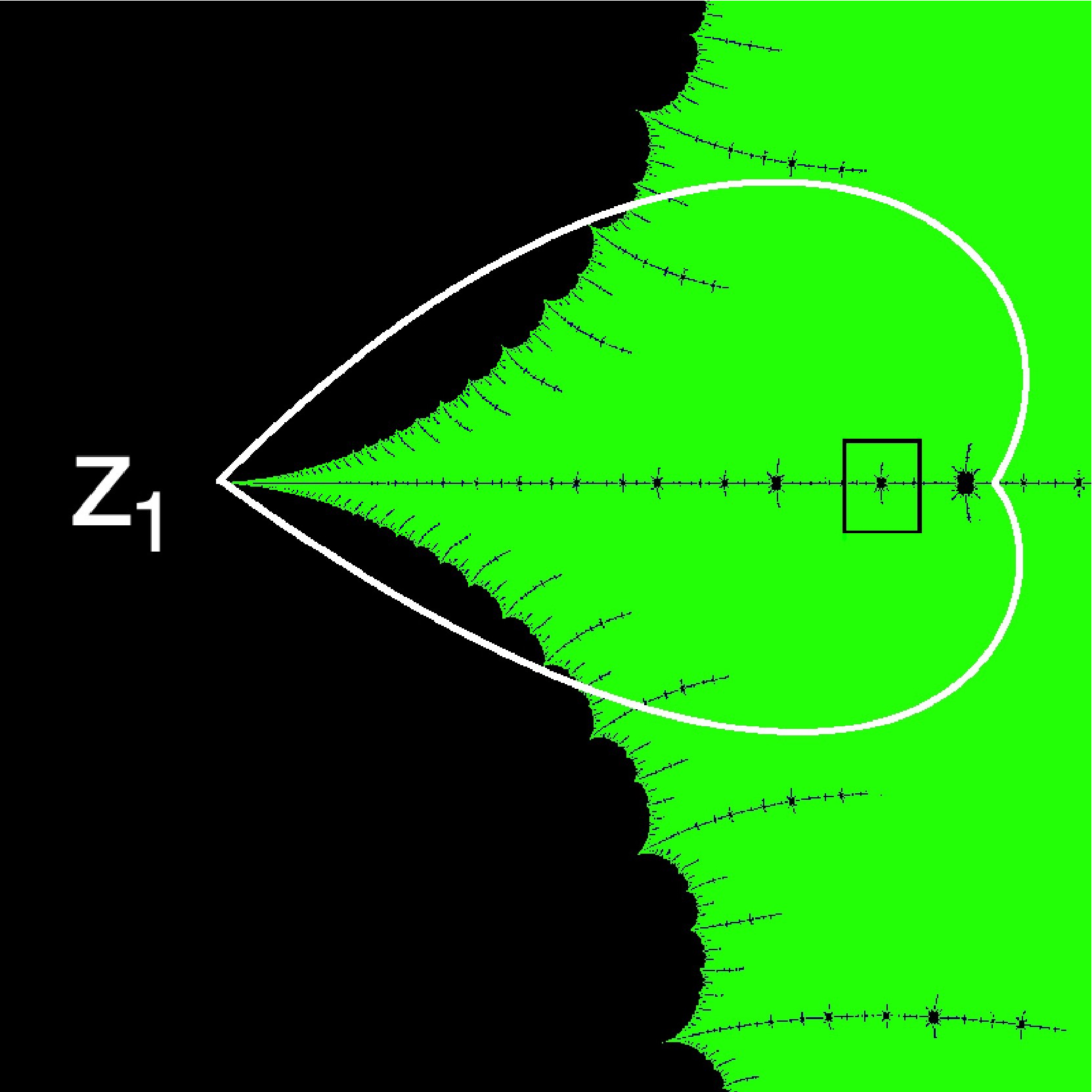}
\end{minipage}
\begin{minipage}{0.48\linewidth}
\includegraphics[scale=0.22]{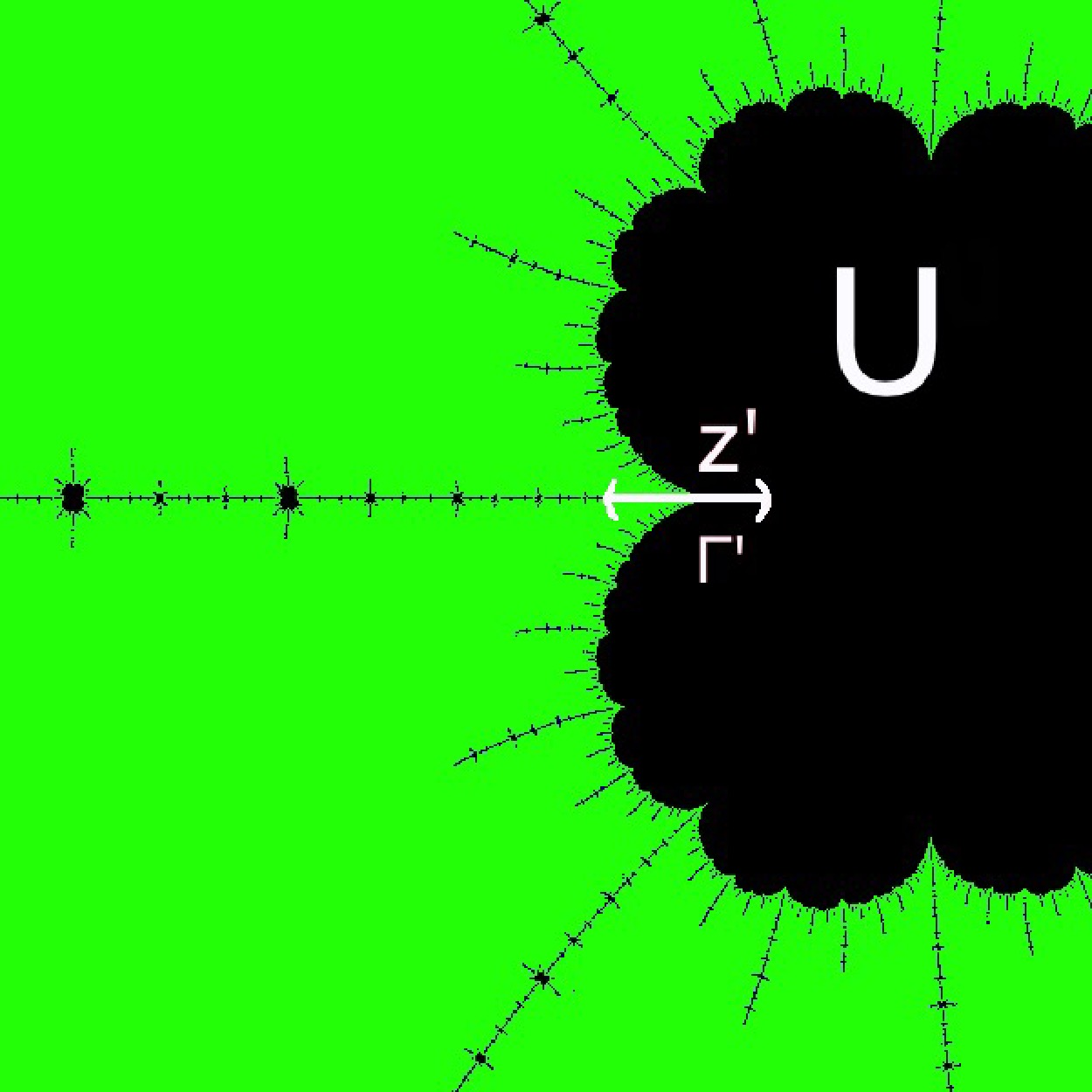}
\end{minipage}
\end{center}
\begin{center}
\begin{minipage}{0.6\linewidth}
\includegraphics[scale=0.27]{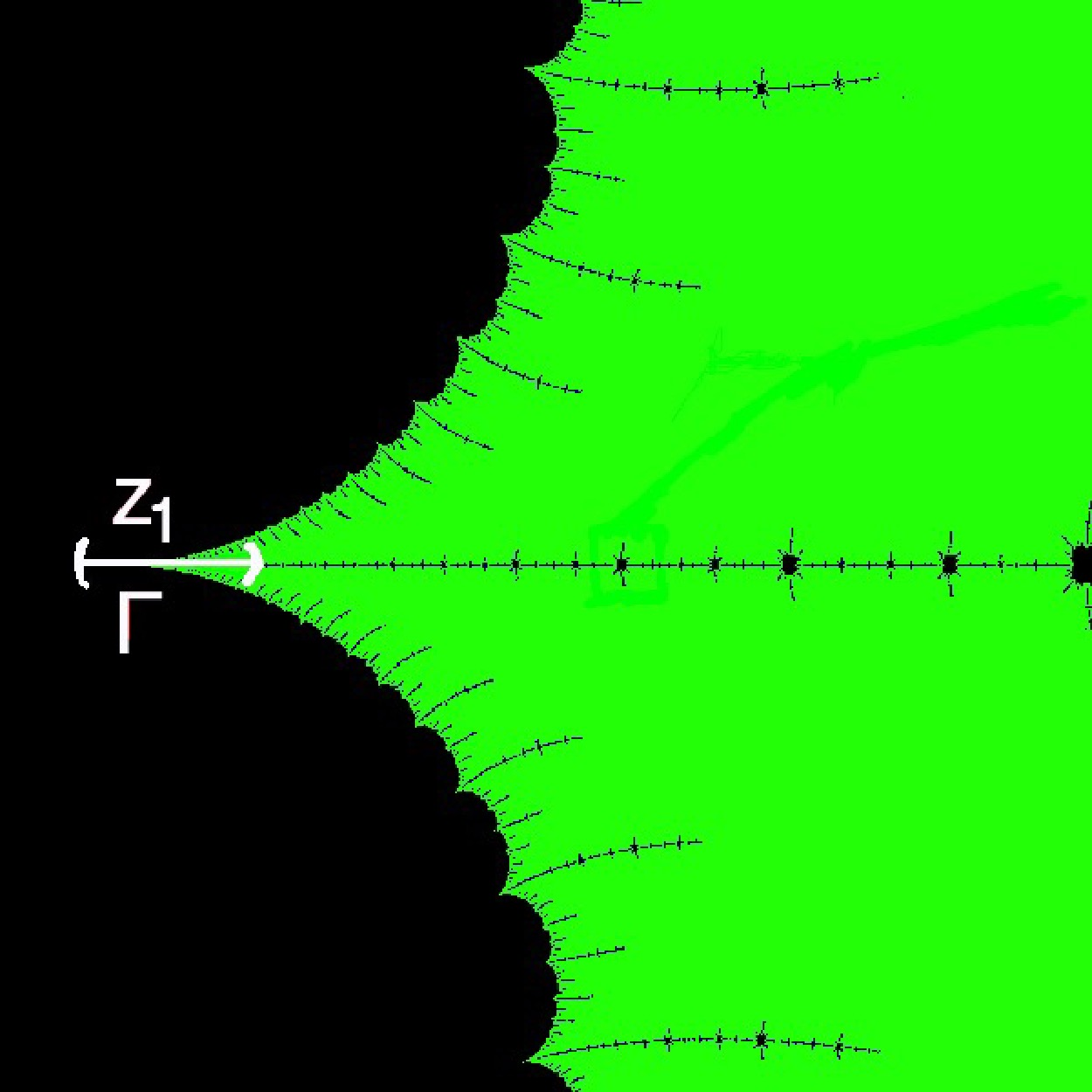}
\end{minipage}
\end{center}
\caption{Top left: A repelling petal at the characteristic parabolic point is enclosed by the white curve. The repelling equator at the characteristic parabolic point is contained in the filled Julia set. The square box contains a Fatou component $U$ such that $z'$ is a point of intersection of $\partial U$ with the repelling equator. Top right: A blow-up of the Fatou component $U$ contained in the box shown in the left figure. Bottom: The piece $\Gamma'$ maps to an invariant analytic arc $\Gamma$ which passes through the characteristic point $z_1$, and lies in the filled Julia set.}
\label{piece returns}
\end{figure}

\begin{proof}
Pick any bounded Fatou component $U$ (different from the characteristic Fatou component) that the repelling equator hits. Assume that the equator intersects $\partial U$ at some pre-parabolic point $z'$. Consider a small piece $\Gamma'$ of the equator with $z'$ in its interior. Since $z'$ eventually falls on the parabolic orbit, some large iterate of $f_c$ maps $z'$ to $z_1$ by a local biholomorphism carrying $\Gamma'$ to an analytic arc $\Gamma$ (say, $\Gamma = f_{\widetilde{c}}^{\circ n} (\Gamma')$)  passing through $z_1$ (compare Figure~\ref{piece returns}). We will show that $\Gamma$ agrees with the repelling equator (up to truncation). Indeed, the repelling equator, and the curve $\Gamma$ are both parts of two loose parabolic trees (any forward iterate of a loose parabolic tree is contained in a loose parabolic tree), and hence must coincide along a Cantor set of points on the Julia set. As analytic arcs, they must thus coincide up to truncation. In particular, the part of $\Gamma$ not contained in the characteristic Fatou component is contained in the repelling equator, and is forward invariant. Straighten the analytic arc $\Gamma$ to an interval $(-\epsilon, \epsilon) \subset \mathbb{R}$ by a local biholomorphism $\alpha : V \rightarrow \mathbb{C}$ such that $z_1 \in V$, and $\alpha(z_1)=0$ (for convenience, we choose $V$ such that it is symmetric with respect to $\Gamma$). This local biholomorphism conjugates the parabolic germ of $f_{\widetilde{c}}^{\circ 2k}$ at $z_1$ to a germ that fixes $0$. Moreover, the conjugated germ maps small enough positive reals to positive reals. Clearly, this must be a real germ. Thus, the parabolic germ of $f_{\widetilde{c}}^{\circ 2k}$ at $z_1$ is analytically conjugate to a real germ.

Observe that $\iota : z \mapsto z^*$ is a topological conjugacy between $f_{\widetilde{c}}$ and $f_{\widetilde{c}^*}$. One can carry out the preceding construction with $f_{\widetilde{c}^*}$, and show that the parabolic germ of $f_{\widetilde{c}^*}^{\circ 2k}$ at $z_1^*$ is analytically conjugate to a real germ. In fact, the role of $\Gamma'$ is now played by $\iota(\Gamma')$, and hence, the role of $\Gamma$ is played by $f_{\widetilde{c}^*}^{\circ n} (\iota(\Gamma')) = \iota(\Gamma)$. Then the biholomorphism $\iota \circ \alpha \circ \iota : \iota(V) \rightarrow \mathbb{C}$ straightens $\iota(\Gamma)$. Conjugating the parabolic germ of $f_{\widetilde{c}^*}^{\circ 2k}$ at $z_1^*$ by $\iota \circ \alpha \circ \iota$, one recovers the same real germ as in the previous paragraph. Thus, the parabolic germs given by the restrictions of $f_{\widetilde{c}}^{\circ 2k}$ and $f_{\widetilde{c}^*}^{\circ 2k}$ at their characteristic parabolic points are analytically conjugate. Moreover, since the critical orbits of $f_{\widetilde{c}}^{\circ 2k}$ lie on the equator, and since the equator is mapped to the real line by $\alpha$, the conjugacy $\eta:=\left(\iota \circ \alpha \circ \iota\right)^{-1}\circ\alpha$ preserves the critical orbits; i.e., it maps $f_{\widetilde{c}}^{\circ kr}(\widetilde{c})$ to $f_{\widetilde{c}^*}^{\circ kr}(\widetilde{c}^*)$ (for $r$ large enough, so that $f_{\widetilde{c}}^{\circ kr}(\widetilde{c})$ is contained in the domain of definition of $\alpha$).
\end{proof} 

\begin{remark}
It follows from the previous lemma that the real-analytic curve $\Gamma$ passing through $z_1$ is invariant under $f_{\widetilde{c}}^{\circ k}$. Indeed, $\Gamma$ is formed by parts of the attracting equator, the repelling equator, and the parabolic point $z_1$. 
\end{remark}

\begin{remark}
Observe that $f_{\widetilde{c}}^{\circ 2k}$ has three critical points, and two (infinite) critical orbits in the characteristic Fatou component $U_{\widetilde{c}}$. Two of these three critical points (of $f_{\widetilde{c}}^{\circ 2k}\vert_{U_{\widetilde{c}}}$) are mapped to the same point by $f_{\widetilde{c}}^{\circ 2k}$ so that they lie on the same critical orbit of $f_{\widetilde{c}}^{\circ 2k}$, and the third one lies on the other critical orbit of  $f_{\widetilde{c}}^{\circ 2k}$. Hence, the two critical orbits of $f_{\widetilde{c}}^{\circ 2k}\vert_{U_{\widetilde{c}}}$ are dynamically distinct.

We want to emphasize the fact that the conjugacy $\eta=\left(\iota \circ \alpha \circ \iota\right)^{-1}\circ\alpha$ constructed (from the condition that an umbilical cord lands at $\widetilde{c}$) in the previous lemma is special: it maps (the tails of) each of the two (dynamically distinct) critical orbits of $f_{\widetilde{c}}^{\circ 2k}$ to (the tails of) the corresponding critical orbits of $f_{\widetilde{c}^*}^{\circ 2k}$. In fact, the parabolic germs given by the restrictions of $f_{\widetilde{c}}^{\circ 2k}$ and $f_{\widetilde{c}^*}^{\circ 2k}$ at their characteristic parabolic points are \emph{always} conformally conjugate by $\iota\circ f_{\widetilde{c}}^{\circ k}$ (note that $\iota$ is an anti-conformal conjugacy between the  restrictions of $f_{\widetilde{c}}^{\circ 2k}$ and $f_{\widetilde{c}^*}^{\circ 2k}$ at their characteristic parabolic points, but  $\iota\circ f_{\widetilde{c}}^{\circ k}$ is a conformal conjugacy between them). However, this local conjugacy exchanges the two post-critical orbits, which have different topological dynamics. Hence this local conjugacy has no chance of being extended to the entire parabolic basin.
\end{remark}

\subsection{A Brief Digression to Parabolic Germs}

We have showed that if an umbilical cord lands at $\widetilde{c}$, then the restriction of $f_{\widetilde{c}}^{\circ 2k}$ at its characteristic parabolic point is analytically conjugate to a real germ. 
Let us denote the local reflection with respect to the curve $\Gamma$ (as in the previous lemma) by $\iota_\Gamma$. Then, $\iota_\Gamma$ commutes with $f_{\widetilde{c}}^{\circ k}$. Therefore, $f_{\widetilde{c}}^{\circ 2k} = \left( f_{\widetilde{c}}^{\circ k} \circ \iota_\Gamma \right) \circ \left(\iota_\Gamma \circ f_{\widetilde{c}}^{\circ k} \right) = g^{\circ 2}$, where $g = f_{\widetilde{c}}^{\circ k} \circ \iota_\Gamma = \iota_\Gamma \circ f_{\widetilde{c}}^{\circ k}$. Thus, the parabolic germ given by the restriction of $f_{\widetilde{c}}^{\circ 2k}$ in a neighborhood of $z_1$ is (locally) the second iterate of a holomorphic germ (which is a holomorphic germ with a parabolic fixed point at $z_1$ and multiplier $1$). 

By the classical theory of conformal conjugacy classes of parabolic germs, there is an infinite-dimensional family of conformally different parabolic germs \cite{Ec1, Vor}. With this in mind, the conclusion of Lemma~\ref{path gives analytically conjugate parabolic germs}, and the properties of the parabolic germ $f_{\widetilde{c}}^{\circ 2k}$ at its characteristic parabolic point discussed in the previous paragraph, seem very unlikely to hold unless the polynomial $f_{\widetilde{c}}^{\circ 2k}$ has a strong global symmetry. In fact, in the next section, we will prove that this can happen if and only if $\widetilde{c}$ lies on the real line or on one of it rotates. The proof, however, depends heavily on the structure of the polynomial $f_{\widetilde{c}}^{\circ 2k}$. It is natural to try to understand the global implications of a local information about a polynomial parabolic germ in more general settings. In particular, we ask the following questions:

\begin{question}[From Germs to Polynomials]\label{local_symmetry_to_global}~
\begin{enumerate}
\item Let $p$ be a complex polynomial with a parabolic fixed point at $0$ with multiplier $1$.
 
\begin{enumerate}
\item If the parabolic germ $p\vert_{B(0,\epsilon)}$ is locally conformally conjugate to a real parabolic germ, or

\item if the parabolic germ $p\vert_{B(0,\epsilon)}$ is locally the second iterate of a holomorphic germ,
\end{enumerate}
can we conclude that the polynomial $p$ has a corresponding global property?

\item Let $p_1$ and $p_2$ be two complex polynomials with parabolic fixed points at $0$ with multiplier $1$. If the two parabolic germs $p_1$ and $p_2$ at the origin are conformally conjugate, are $p_1$ and $p_2$ globally semi-conjugate or semi-conjugate from/to a common polynomial?
\end{enumerate}
\end{question}

Conditions (1a) and (1b) on the parabolic germ of $p$ translate into corresponding conditions on its extended horn maps (such that its domain is maximal). Indeed, Condition (1a) is equivalent to saying that the (suitably normalized) horn maps of $p\vert_{B(0,\epsilon)}$ at $0$ and $\infty$ are conjugate under the involution $1/\overline{w}$; and Condition (1b) is equivalent to the statement that the horn map of $p\vert_{B(0,\epsilon)}$ at $0$ commutes with $w\mapsto-w$ \cite[\S 2]{FL}. On the other hand, Condition (2) is equivalent to having the same horn maps for $p_1$ and $p_2$ at $0$, up to pre and post composition by multiplications with non-zero complex numbers.

We provide partial answers to some of these local-global questions in a sequel to this work \cite{IM5}. 

\section{Wiggling of Umbilical Cords}\label{conjugacy_extension}
The main goal of this section is to prove that umbilical cords never land away from the real line or its rotates. More precisely, we will show that the existence of a path as in Lemma~\ref{straight_equator} would imply that $\widetilde{c} \in \mathbb{R} \cup \omega\mathbb{R} \cup \omega^2\mathbb{R} \cup \cdots \cup \omega^{d}\mathbb{R}$, where $\omega=e^{\frac{2\pi i}{d+1}}$. To this end, we will first extend the local analytic conjugacy between the parabolic germs, constructed in Lemma~\ref{path gives analytically conjugate parabolic germs}, to a conformal conjugacy between polynomial-like restrictions. This would allow us to apply a theorem of \cite{I2} to deduce that the maps $f_{\widetilde{c}}^{\circ 2k}$ and $f_{\widetilde{c}^*}^{\circ 2k}$ are conjugate by an irreducible holomorphic correspondence. In other words, we will show that $f_{\widetilde{c}}^{\circ 2k}$ and $f_{\widetilde{c}^*}^{\circ 2k}$ are polynomially semi-conjugate to a common polynomial $p$. The final step involves proving that $f_{\widetilde{c}}^{\circ 2k}$ and $f_{\widetilde{c}^*}^{\circ 2k}$ are, in fact, affinely conjugate.

Recall that $\widetilde{c}$ is the Ecalle height $0$ parameter on the root parabolic arc $\mathcal{C}$ of the hyperbolic component $H$ (of period $k$), $z_1$ is its characteristic parabolic point, and $U_{\widetilde{c}}$ is its characteristic Fatou component. We choose a Riemann map $\phi_{\widetilde{c}}$ of $U_{\widetilde{c}}$ normalized so that it sends the critical value $\widetilde{c}$ to $0$, and its homeomorphic extension to the boundary sends the parabolic point on $\partial U_{\widetilde{c}}$ to $1$. Then, $\phi_{\widetilde{c}}$ conjugates the first holomorphic return map $f_{\widetilde{c}}^{\circ 2k}$ on $U_{\widetilde{c}}$ to a Blaschke product $B$. Furthermore, the Ecalle height $0$ condition implies that the images of the two critical orbits of $f_{\widetilde{c}}^{\circ 2k}$ under an attracting Fatou coordinate are related by translation by $1/2$. It follows that the local compositional square root of $f_{\widetilde{c}}^{\circ 2k}$ in an attracting petal (i.e., translation by $1/2$ pulled back by the Fatou coordinate) can be analytically extended throughout the immediate basin $U_{\widetilde{c}}$. Therefore, the first return map $f_{\widetilde{c}}^{\circ 2k}$ on $U_{\widetilde{c}}$ is the second iterate of a holomorphic map $g$ preserving $U_{\widetilde{c}}$ with a parabolic fixed point of multiplier $+1$ at $z_1$. Since $f_{\widetilde{c}}^{\circ 2k}$ has two critical values in $U_{\widetilde{c}}$, $g$ is unicritical. Hence, the Riemann map $\phi_{\widetilde{c}}$ conjugates $g$ to the Blaschke product $\widetilde{B}(w) = \frac{(d+1)w^d+(d-1)}{(d+1)+(d-1)w^d}$. It follows that $B(w) = \widetilde{B}^{\circ 2}(w)$. We record this fact in the following lemma.

\begin{lemma}[Ecalle Height Zero Basins Are Conformally Conjugate]
The first holomorphic return map of an immediate basin of a parabolic point of odd period of a critical Ecalle height $0$ parameter is conformally conjugate to $B(w) = \widetilde{B}^{\circ 2}(w)$, where $\widetilde{B}(w) = \frac{(d+1)w^d+(d-1)}{(d+1)+(d-1)w^d}$. In particular, they are conformally conjugate to each other.
\end{lemma}

At this point, we know that $f_{\widetilde{c}}^{\circ 2k}$ and $f_{\widetilde{c}^*}^{\circ 2k}$, restricted to the characteristic Fatou components, are conformally conjugate, and the corresponding parabolic germs are also conformally conjugate by a `critical orbit'-preserving local biholomorphism. The next lemma essentially shows that these two conjugacies can be glued together.

Recall that in Lemma~\ref{path gives analytically conjugate parabolic germs}, we constructed a conformal conjugacy $\eta$ between the parabolic germs given by the restrictions of $f_{\widetilde{c}}^{\circ 2k}$ and $f_{\widetilde{c}^*}^{\circ 2k}$ at their characteristic parabolic points. Moreover, $\eta$ maps $f_{\widetilde{c}}^{\circ kr}(\widetilde{c})$ to $f_{\widetilde{c}^*}^{\circ kr}(\widetilde{c}^*)$, for $r$ large enough. 

\begin{lemma}[Extension to The Immediate Basin]\label{extension_to_basin}
The conformal conjugacy $\eta$ between the parabolic germs of $f_{\widetilde{c}}^{\circ 2k}$ and $f_{\widetilde{c}^*}^{\circ 2k}$ at their characteristic parabolic points can be extended to a conformal conjugation between the dynamics in the immediate basins.
\end{lemma}

\begin{proof}
We need to choose our conformal change of coordinates symmetrically for $f_{\widetilde{c}}$ and $f_{\widetilde{c}^*}$. 

Let us choose the attracting Fatou coordinate $\psi^{\mathrm{att}}_{\widetilde{c}}$ in $U_{\widetilde{c}}$ normalized so that it maps the equator to the real line, and $\psi^{\mathrm{att}}_{\widetilde{c}}(\widetilde{c}) = 0$. This naturally determines our preferred attracting Fatou coordinate $\psi^{\mathrm{att}}_{\widetilde{c}^*} := \iota \circ \psi^{\mathrm{att}}_{\widetilde{c}} \circ \iota$ for $f_{\widetilde{c}^*}$ at its characteristic parabolic point $z_1^*$, and we have $\psi^{\mathrm{att}}_{\widetilde{c}^*}(\widetilde{c}^*) =0$. 

Recall that we constructed a conformal conjugacy $\eta= \iota \circ \alpha^{-1} \circ \iota \circ \alpha : V \rightarrow \iota(V)$ between the parabolic germs $f_{\widetilde{c}}^{\circ 2k}$ and $f_{\widetilde{c}^*}^{\circ 2k}$ at their characteristic parabolic points in Lemma~\ref{path gives analytically conjugate parabolic germs}. $\eta$ maps some attracting petal (not necessarily containing $\widetilde{c}$) $P\subset V$ of $f_{\widetilde{c}}^{\circ 2k}$ at $z_1$ to some attracting petal $\iota(P) \subset \iota(V)$ of $f_{\widetilde{c}^*}^{\circ 2k}$ at $z_1^*$. Hence, $\psi^{\mathrm{att}}_{\widetilde{c}} \circ \eta^{-1}$ is an attracting Fatou coordinate for $f_{\widetilde{c}^*}^{\circ 2k}$ at $z_1^*$. By the uniqueness of Fatou coordinates, $\psi^{\mathrm{att}}_{\widetilde{c}} \circ \eta^{-1}(z) = \psi^{\mathrm{att}}_{\widetilde{c}^*}(z) + a$, for some $a \in \mathbb{C}$, and for all $z$ in their common domain of definition. There is some large $n$ for which $f_{\widetilde{c}^*}^{\circ 2kn}(\widetilde{c}^*)$ belongs to $\iota(V)$, the domain of definition of $\eta^{-1}$. By definition, 
$$\psi^{\mathrm{att}}_{\widetilde{c}} \circ \eta^{-1}(f_{\widetilde{c}^*}^{\circ 2kn}(\widetilde{c}^*))= \psi^{\mathrm{att}}_{\widetilde{c}} \circ \alpha ^{-1} \circ \iota \circ \alpha \circ \iota \circ f_{\widetilde{c}^*}^{\circ 2kn}(\widetilde{c}^*)$$
$$\hspace{8mm} =\psi^{\mathrm{att}}_{\widetilde{c}} \circ \alpha ^{-1} \circ \iota \circ \alpha \circ f_{\widetilde{c}}^{\circ 2kn}\circ \iota (\widetilde{c}^*)= \psi^{\mathrm{att}}_{\widetilde{c}} \left(\alpha^{-1}\left( \iota\left(\alpha\left(f_{\widetilde{c}}^{\circ 2kn}(\widetilde{c})\right)\right)\right)\right)$$
$$\hspace{-1.5cm} =\psi^{\mathrm{att}}_{\widetilde{c}} \left(\alpha^{-1}\left(\alpha\left(f_{\widetilde{c}}^{\circ 2kn}(\widetilde{c})\right)\right)\right)
=\psi^{\mathrm{att}}_{\widetilde{c}} \left(f_{\widetilde{c}}^{\circ 2nk} (\widetilde{c})\right)=n.$$

This holds since $\widetilde{c}$ lies on the equator, and $\alpha$ maps $f_{\widetilde{c}}^{\circ 2nk} (\widetilde{c})$ to the real line. But, $$\psi^{\mathrm{att}}_{\widetilde{c}^*}(f_{\widetilde{c}^*}^{\circ 2kn}(\widetilde{c}^*))= \iota \circ \psi^{\mathrm{att}}_{\widetilde{c}} \circ \iota \left( \left( f_{\widetilde{c}}^{\circ 2nk}(\widetilde{c})\right)^{*} \right)= n.$$

This shows that $a=0$, and hence, $\eta = \left(\psi^{\mathrm{att}}_{\widetilde{c}^*}\right)^{-1} \circ \psi^{\mathrm{att}}_{\widetilde{c}}$ on $P$.

\begin{figure}[ht!]
\begin{minipage}{0.48\linewidth}
\begin{center}
\includegraphics[scale=0.25]{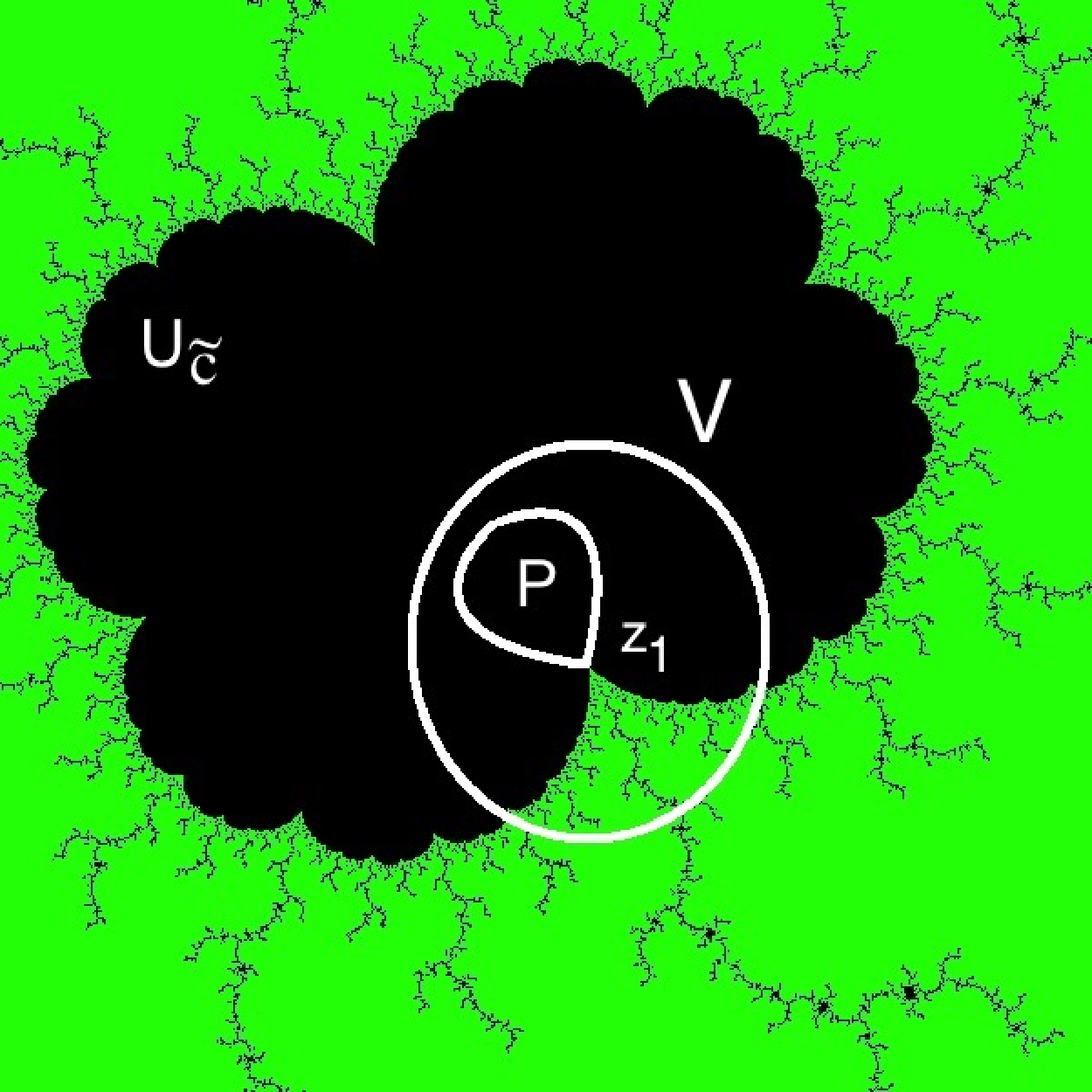}
\end{center}
\end{minipage}
\begin{minipage}{0.48\linewidth}
\begin{center}
\includegraphics[scale=0.25]{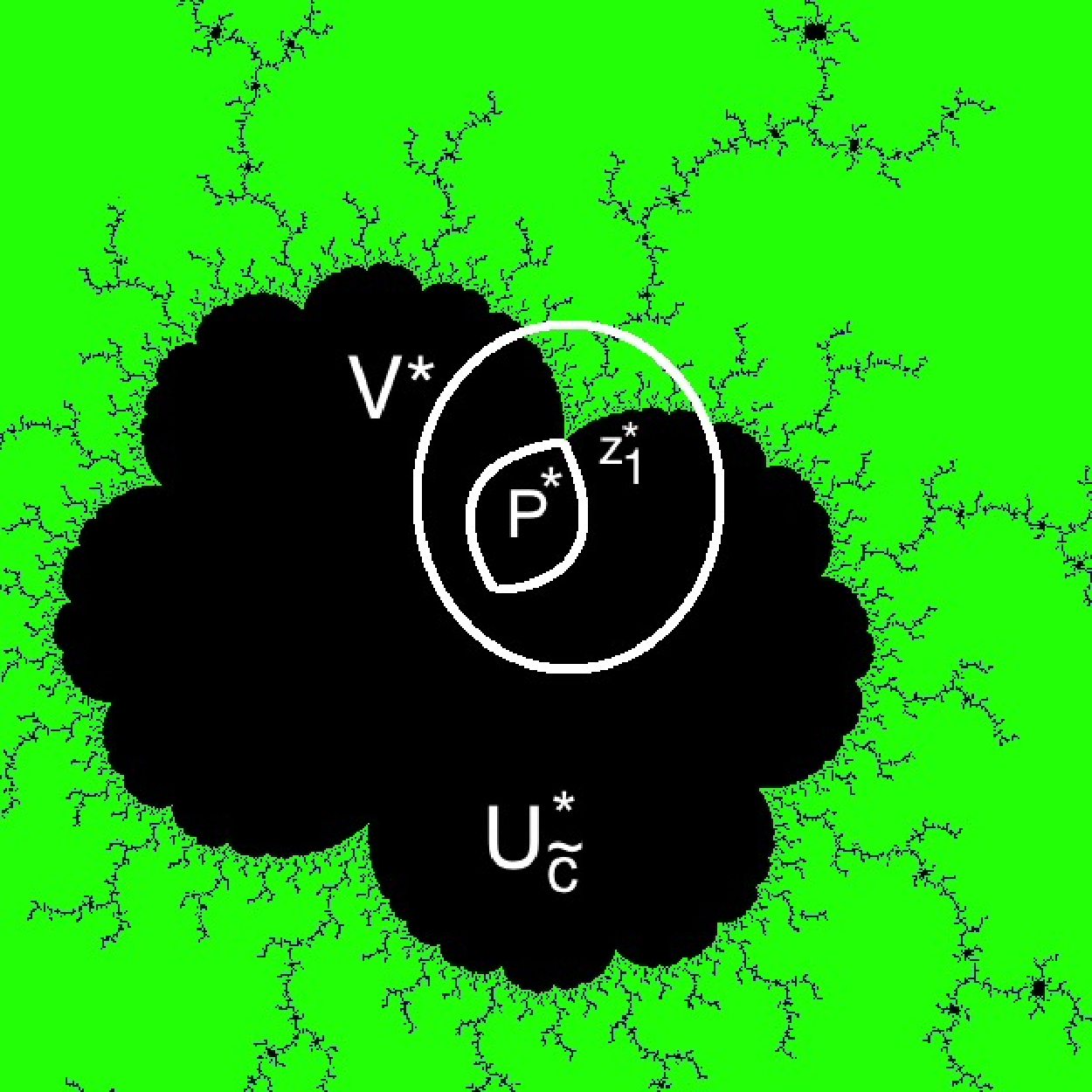}
\end{center}
\end{minipage}
\caption{The germ conjugacy $\eta: V \rightarrow \iota(V)$, and the basin conjugacy $\chi : U_{\widetilde{c}} \rightarrow \iota(U_{\widetilde{c}})$ agree with $\left(\psi^{\mathrm{att}}_{\widetilde{c}^*}\right)^{-1} \circ \psi^{\mathrm{att}}_{\widetilde{c}}: P\rightarrow \iota(P)$.}
\label{conjugacy_extends}
\end{figure}

We also fix the Riemann map $\phi_{\widetilde{c}}: U_{\widetilde{c}} \rightarrow \mathbb{D}$ which conjugates $f_{\widetilde{c}}^{\circ 2k}$ on $U_{\widetilde{c}}$ to the Blaschke product $B$ in the previous lemma. Since the immediate basin of $f_{\widetilde{c}^*}$ at its characteristic parabolic point $z_1^*$ is $\iota(U_{\widetilde{c}})$, $\iota \circ \phi_{\widetilde{c}} \circ \iota$ is the preferred Riemann map of the basin that sends the critical value $\widetilde{c}^*$ to $0$, sends the parabolic point $z_1^*$ to $1$, and conjugates $f_{\widetilde{c}^*}^{\circ 2k}$ to the Blaschke product $B$. A similar argument as above shows that the isomorphism $\chi:= \iota \circ \phi_{\widetilde{c}}^{-1} \circ \iota \circ \phi_{\widetilde{c}}: U_{\widetilde{c}} \rightarrow \iota(U_{\widetilde{c}})$ agrees with $\left(\psi^{\mathrm{att}}_{\widetilde{c}^*}\right)^{-1} \circ \psi^{\mathrm{att}}_{\widetilde{c}}$ on $P$, and hence extends the local conjugacy $\eta$ to the entire immediate basin $U_{\widetilde{c}}$ (compare Figure~\ref{conjugacy_extends}), such that it conjugates $f_{\widetilde{c}}^{\circ 2k}$ on $U_{\widetilde{c}}$ to $f_{\widetilde{c}^*}^{\circ 2k}$ on $\iota(U_{\widetilde{c}})$.
\end{proof}

Abusing notation, let us denote the extended conjugacy from $U_{\widetilde{c}} \cup V$ onto $\iota\left(U_{\widetilde{c}} \cup V\right)$ of the previous lemma by $\eta$. Our next goal is to extend $\eta$ to a neighborhood of $\overline{U_{\widetilde{c}}}$ (the topological closure of $U_{\widetilde{c}}$). 

\begin{lemma}[Extension to a Neighborhood of The Basin Closure]\label{extension_to_basin_closure}
$\eta$ can be extended conformally to a  neighborhood of $\overline{U_{\widetilde{c}}}$.
\end{lemma}

\begin{proof}
Observe that the basin boundaries are locally connected, and hence by Carath\'eodory's theorem, the conformal conjugacy $\eta$ extends as a homeomorphism from $\partial U_{\widetilde{c}}$ onto $\partial \iota(U_{\widetilde{c}})$. Moreover, $\eta$ extends analytically across the point $z_1$. At this point, the existence of the required extension follows from \cite[Lemma~2, Lemma~3]{BE1}. However, we have a more straightforward proof (essentially using the same idea) as our maps are unbranched on the Julia set.

By Montel's theorem, $\bigcup_{n} f_{\widetilde{c}}^{\circ 2kn}\left( V \cap \partial U_{\widetilde{c}} \right) = \partial U_{\widetilde{c}}$. As none of the $f_{\widetilde{c}}^{\circ 2kn}$ have critical points on $\partial U_{\widetilde{c}}$, we can extend $\eta$ in a neighborhood of each point of $\partial U_{\widetilde{c}}$ by simply using the equation $\eta \circ f_{\widetilde{c}}^{\circ 2kn} = f_{\widetilde{c}^*}^{\circ 2kn} \circ \eta$. Since all of these extensions at various points of $\partial U_{\widetilde{c}}$ extend the already defined (and conformal) common map $\eta$, the uniqueness of analytic continuations yields an analytic extension of $\eta$ in a neighborhood of $\overline{U_{\widetilde{c}}}$. By construction, this extension is clearly a proper holomorphic map, and assumes every point in $\iota(U_{\widetilde{c}})$ precisely once. Therefore, the extended $\eta$ from a neighborhood of $\overline{U_{\widetilde{c}}}$ onto a neighborhood of $\overline{\iota(U_{\widetilde{c}})}$ has degree $1$, and is a conformal conjugacy between $f_{\widetilde{c}}^{\circ 2k}$ and $f_{\widetilde{c}^*}^{\circ 2k}$.
\end{proof}

We are now ready to apply the `local to global' result from \cite{I2}.
 
\begin{lemma}[Global Semi-conjugacy]\label{global semi-conjugacy}
There exist polynomials $p$, $p_1$ and $p_2$ such that $f_{\widetilde{c}}^{\circ 2k} \circ p_1 = p_1 \circ p$, $f_{\widetilde{c}^*}^{\circ 2k}\circ p_2 = p_2 \circ p$, and $\Deg p_1 = \Deg p_2$.
\end{lemma}

\begin{proof}
Note that $f_{\widetilde{c}}^{\circ 2k}$ (respectively $f_{\widetilde{c}^*}^{\circ 2k}$) restricted to a small neighborhood of $\overline{U_{\widetilde{c}}}$ (respectively $\overline{\iota(U_{\widetilde{c}})}$) is polynomial-like of degree $d^2$, and it follows from the previous lemma that these two polynomial-like maps are conformally conjugate. Applying \cite[Theorem~1]{I2} to this situation, we obtain the existence of polynomials $p$, $p_1$, and $p_2$ such that the required semi-conjugacies hold. Since $f_{\widetilde{c}}^{\circ 2k}$ and $f_{\widetilde{c}^*}^{\circ 2k}$ are topologically conjugate by $\iota$, it follows from the proof of \cite[Theorem~1]{I2} that $\Deg p_1 = \Deg p_2$ (observe that the product dynamics $\left(f_{\widetilde{c}}^{\circ 2k},\ f_{\widetilde{c}^*}^{\circ 2k}\right)$ is globally self-conjugate by $(\iota\times\iota)\circ q$, where $q:\mathbb{C}^2\rightarrow\mathbb{C}^2,\ q(z,\ w)=(w,\ z)$).
\end{proof}

In order to finish the proof of Theorem~\ref{umbilical_cord_wiggling}, we need to use a classification of semi-conjugate polynomials proved in \cite[Appendix~A]{I2}. The results are based on the work of Ritt and Engstrom.

Let $\mathcal{S}$ be the set of all affine conjugacy classes of triples $\left(f, g, h\right)$ of polynomials of degree at least two such that $f \circ h = h \circ g$, where we say that two triples $(f_1, g_1, h_1)$ and $(f_2, g_2, h_2)$ are affinely conjugate if there exist affine maps $\sigma_1, \sigma_2$ such that $$f_2 = \sigma_1 \circ f_1 \circ \sigma_1^{-1},\ g_2 = \sigma_2 \circ g_1 \circ \sigma_2^{-1},\ \mathrm{and}\ h_2 = \sigma_1 \circ h_1 \circ \sigma_2^{-1}.$$

We denote $(f_1, g_1, h_1) \sim (f_2, g_2, h_2)$. 

The following theorem tells us that one can always apply a reduction step to assume that $\Deg f = \Deg g$ and $\Deg h$ are co-prime.

\begin{theorem}\label{reduction}
Let $\left[(f, g, h)\right] \in \mathcal{S}$. If $\gcd(\Deg f, \Deg h) = d > 1$, then there exist polynomials $g_1, h_1, f_1, \widehat{h}_1, \alpha_1$ and $\beta_1$ such that
\begin{align*}
 f \circ h_1 &= h_1 \circ g_1, & f_1 \circ \widehat{h}_1 &= \widehat{h}_1 \circ g, &
 h &= h_1 \circ \beta_1 = \alpha_1 \circ \widehat{h}_1, 
\end{align*}
\begin{align*}
  \Deg f &= \Deg g_1 = \Deg f_1, &\Deg \alpha_1 &= \Deg \beta_1 = d, &\Deg h_1 &= \Deg \widehat{h}_1 = \Deg h/d.
 \end{align*}

In particular, if $d < \Deg h$, then $\left[(f, g_1, h_1)\right], [(f_1, g, \widehat{h}_1)] \in \mathcal{S}$.
\end{theorem} 

The next theorem gives a complete classification for the case $\gcd(\Deg f, \Deg h) = 1$.

\begin{theorem}\label{semiconjugacy_classification}
Assume that $[(f, g, h)] \in \mathcal{S}$ satisfies $\gcd(\Deg f, \Deg h) = 1$. Then there exists a representative $(f_0, g_0, h_0)$ of $[(f, g, h)]$ such that one of the following is true:
\begin{itemize}
\item $f(z) = z^r \left(P(z)\right)^b$, $g(z) = z^r P(z^b)$ and $h(z) = z^b$, where $r = a\ mod\ b$, and $P$ is a complex polynomial,
\item $f = g = T_a,\ h = T_b$ are Chebyshev polynomials (of degree $a$ and $b$ respectively), 
\end{itemize}
where $a = \Deg f (= \Deg g)$ and $b = \Deg h$.
\end{theorem} 

We are now ready to complete the proof of Theorem~\ref{umbilical_cord_wiggling}.

\begin{proof}[Proof of Theorem~\ref{umbilical_cord_wiggling}]
If the two semi-conjugacies appearing in Lemma~\ref{global semi-conjugacy} are affine conjugacies (i.e., if $p_1$ and $p_2$ have degree $1$), then $f_{\widetilde{c}}^{\circ 2k}$ and $f_{\widetilde{c}^*}^{\circ 2k}$ are affinely conjugate, and a straightforward computation shows that $\widetilde{c}^*=\omega^j \widetilde{c}$ where $\omega=\exp(\frac{2\pi i}{d+1})$, and $j \in \mathbb{N}$. Setting $\widetilde{c}=r e^{2\pi i \theta}$, we see that $\theta \in \mathbb{Z}/2(d+1)$. We need to consider two cases now. When $d$ is even, the condition $\theta \in \mathbb{Z}/2(d+1)$ implies that there is some integer $k$ such that either $\theta + \frac{k}{d+1} =0$ or $\theta + \frac{k}{d+1} =\frac{1}{2}$. Therefore when $d$ is even, $f_{\widetilde{c}}$ is affinely conjugate to some $f_c$ with $c\in \mathbb{R}$. Now let us consider the case when $d$ is odd. In this case, the condition $\theta \in \mathbb{Z}/2(d+1)$ does not necessarily imply that $f_{\widetilde{c}}$ is affinely conjugate to some $f_c$ with $c\in \mathbb{R}$. However for odd degree multicorns, the situation is rather restricted. If $f_{\widetilde{c}}$ has a parabolic cycle for some $\widetilde{c}=r e^{2\pi i \theta}$ with $\theta \in \mathbb{Z}/(d+1)$ (i.e., $f_{\widetilde{c}}$ is affinely conjugate to a real anti-holomorphic polynomial), then $\widetilde{c}$ lies on a period $1$ parabolic arc. Recall that by \cite[Lemma~5.3]{MNS}, each parabolic arc of period $1$ is a co-root arc, and hence $\widetilde{c}$ cannot be the landing point of a path $p : [0,\delta] \rightarrow \mathbb{C}$ with $p(0) = \widetilde{c}$, and $p((0,\delta]) \subset \mathcal{M}_d^* \setminus \overline{H_0}$ (where $H_0$ is the hyperbolic component of period $1$). On the other hand, if $f_{\widetilde{c}}$ has a parabolic cycle for some $\widetilde{c}=r e^{2\pi i \theta}$ with $\theta \in \mathbb{Z}/2(d+1)\setminus \mathbb{Z}/(d+1)$, then $\widetilde{c}$ is either a parabolic cusp of period $1$, or a co-root point of a hyperbolic component of period $2$. In particular, such a $\widetilde{c}$ cannot lie on a root arc of an odd period hyperbolic component of $\mathcal{M}_d^*$. This completes the proof of the theorem in the case when both $p_1$ and $p_2$ are of degree $1$.

Therefore, we only need to deal with the situation $\Deg p_1 = \Deg p_2  = b > 1$. We will first prove by contradiction that $\gcd(\Deg f_{\widetilde{c}}^{\circ 2k}, \Deg p_1) > 1$. To do this, let $\gcd(\Deg f_{\widetilde{c}}^{\circ 2k}, \Deg p_1) = 1$, i.e., $\gcd (d^{2k}, b) = 1$, i.e., $\gcd(d, b)=1$. Now we can apply Theorem~\ref{semiconjugacy_classification} to our situation; but since $f_{\widetilde{c}}^{\circ 2k}$ is parabolic, it is neither a power map, nor a Chebyshev polynomial. Hence, there exists some non-constant polynomial $P$ such that $f_{\widetilde{c}}^{\circ 2k}$ is affinely conjugate to the polynomial $g(z):=z^r(P(z))^b$. If $r\geq 2$, then $g(z)$ has a super-attracting fixed point at $0$. But $f_{\widetilde{c}}^{\circ 2k}$, which is affinely conjugate to $g(z)$, has no super-attracting fixed point. Hence, $r =0$ or $1$. By degree consideration, we have $d^{2k} = r + bk$, where $\Deg P = k$. The assumption $\gcd(d, b)=1$ implies that $r=1$, i.e., $g(z)=z(P(z))^b$. Now the fixed point $0$ for $g$ satisfies $g^{-1}(0)=\{ 0\} \cup P^{-1}(0)$, and any point in $P^{-1}(0)$ has a local mapping degree $b$ under $g$. The same must hold for the affinely conjugate polynomial $f_{\widetilde{c}}^{\circ 2k}$: there exists a fixed point (say $x$) for $f_{\widetilde{c}}^{\circ 2k}$ such that any point in $(f_{\widetilde{c}}^{\circ 2k})^{-1}(x)$ has mapping degree $b$ except for $x$; in particular, all points in $(f_{\widetilde{c}}^{\circ 2k})^{-1}(x)\setminus \{ x \}$ are critical points for $f_{\widetilde{c}}^{\circ 2k}$ (since $b>1$). However, the local degree of any critical point for $f_{\widetilde{c}}^{\circ 2k}$ is equal to $d^r$ for some $r\geq 1$ (every critical point of $f_{\widetilde{c}}^{\circ 2}$ has mapping degree $d$); so $b=d^r$, and this contradicts the assumption that $\gcd(d,b)=1$ (alternatively, we could use the fact that $f_{\widetilde{c}}^{\circ 2k}$ has no finite critical orbit). 

Applying Engstrom's theorem \cite{Eng} (compare \cite[Theorem~11, Corollary~12, Lemma~13]{I2}), we obtain the existence of polynomials (of degree at least two) $g$, $h$, $g_1$ and $h_1$ such that up to affine conjugacy
\begin{align*}
f_{\widetilde{c}}^{\circ 2k} &= h \circ g,& f_{\widetilde{c}^*}^{\circ 2k} &= h_1 \circ g_1,& p &= g \circ h = g_1 \circ h_1,\ \mathrm{and}\ \Deg (g) =\Deg (g_1).
\end{align*}

The equation $f_{\widetilde{c}}^{\circ 2k} = h \circ g$ implies that $f_{\widetilde{c}^*}^{\circ 2k} = \left(\iota \circ h \circ \iota \right) \circ \left(\iota \circ g \circ \iota\right)$. Note that the only possible non-uniqueness in the decomposition of $f_{\widetilde{c}^*}^{\circ 2k} $ into prime factors (under composition) occurs due to the relation 
\begin{align*}
z^d+\widetilde{c} &= (z^{d_1}+\widetilde{c}) \circ (z^{d_2}) = (z^{d_2}+\widetilde{c}) \circ (z^{d_1})\ \mathrm{with}\ d = d_1 d_2,\ d_1, d_2 \geq 2.
\end{align*}

However, we claim that $h_1 = \iota \circ h \circ \iota$, and $g_1 = \iota \circ g \circ \iota$. Indeed, if $h_1$ and $\iota \circ h \circ \iota$ (and hence $g_1$ and $\iota \circ g \circ \iota$) have different decompositions, then using the type of non-uniqueness, the relation $p = g \circ h = g_1 \circ h_1$, and the fact that $\Deg (\iota\circ g\circ\iota)=\Deg (g)=\Deg (g_1)$, one obtains two different sets of multiplicities of the critical points for the same polynomial $p$. This contradiction proves the claim.

Therefore, $p = g \circ h = (\iota \circ g \circ \iota) \circ (\iota \circ h \circ \iota)$ (up to affine conjugacy). Hence $g$ and $h$ are real polynomials, implying that $g_1=g$, and $h_1=h$. It now follows that $f_{\widetilde{c}}^{\circ 2k}$ and $f_{\widetilde{c}^*}^{\circ 2k}$ are affinely conjugate. We can now argue as in the first paragraph of the proof to conclude that $d$ is even, and  $\widetilde{c} \in \mathbb{R} \cup \omega\mathbb{R} \cup \omega^2\mathbb{R} \cup \cdots \cup \omega^{d}\mathbb{R}$, where $\omega=\exp(\frac{2\pi i}{d+1})$.
\end{proof}

\begin{corollary}
For odd $d$, \emph{all} umbilical cords of $\mathcal{M}_d^*$ wiggle.
\end{corollary}

\begin{corollary}\label{height_zero_real}
Let $c$ be an odd period non-cusp parabolic parameter of $\mathcal{M}_d^*$ with critical Ecalle height $0$. If the characteristic parabolic germ of $f_c$ is conformally conjugate to a real parabolic germ, then $c^*= \omega^j c$ for some $j\in \lbrace 0, 1, \cdots, d\rbrace$, where $\omega=\exp(\frac{2\pi i}{d+1})$. In other words, $f_c$ commutes with the global anti-holomorphic involution $\zeta \mapsto \omega^{-j} \zeta^*$.
\end{corollary}

\begin{remark}
1) We should point out that Theorem~\ref{umbilical_cord_wiggling} shows that if a hyperbolic component $H$ of odd period (different from $1$) of $\mathcal{M}_d^*$ does not intersect the real line or its $\omega$-rotates, then $H$ cannot be connected to the principal hyperbolic component (of period $1$) by a path inside of $\mathcal{M}_d^*$. This statement is a sharper version of a result of Hubbard and Schleicher on the non path-connectedness of the multicorns \cite[Theorem~6.2]{HS}.

2) Corollary~\ref{height_zero_real} is generalized in \cite[\S 4]{IM5}. More precisely, it is shown there that real parabolic germs can only come from real anti-polynomials, without any assumption on the period of the parabolic cycle and its critical Ecalle height.
\end{remark}

\section{Renormalization, and Multicorn-Like Sets}\label{renormalization}
In this section, we will give a brief overview of the combinatorics and topology of straightening maps in consistence with \cite{IK}. After preparing the necessary background on renormalization and straightening, we will introduce the concepts of `multicorn-like sets', and the `straightening map' from `multicorn-like sets' to the actual multicorn (compare \cite{I1}). Finally, we will state the principal results of \cite{IK}, applied to our setting.

\begin{definition}[Polynomial-like and Anti-polynomial-like Maps]\label{DefAnti-PolyLike}
 We call a map $g: U^{\prime} \rightarrow U$ \emph{polynomial-like} (respectively \emph{anti-polynomial-like}) if
\begin{itemize}
\item $U^{\prime}$, $U$ are topological disks in $\mathbb{C}$, and $U^{\prime}$ is relatively compact in $U$.

\item $g : U^{\prime} \rightarrow U$ is holomorphic (respectively anti-holomorphic), and proper.
\end{itemize} 
\end{definition}
  
The \emph{filled Julia} set $K(g)$ and the \emph{Julia set} $J(g)$ are defined as follows:
\begin{center}
$K(g)=\lbrace z\in U^{\prime} :  g^{\circ n}(z) \in U', \hspace{1mm} \forall \hspace{1mm} n \in \mathbb{N} \rbrace$, $J(g)=\partial K(g)$.
\end{center} 

In particular, we say that polynomial-like (or anti-polynomial-like) mapping is \emph{unicritical-like} if it has a unique critical point of possibly higher multiplicity. The importance of (anti-)polynomial-like maps stems from the fact that they behave, in a certain sense, like (anti-)polynomials. This is justified by the following Straightening Theorem \cite[Theorem~1]{DH2} which is proved in the same way as in the holomorphic case: every anti-polynomial-like map of degree $d$ is hybrid equivalent to an anti-polynomial of equal degree. 

\begin{definition}[Hybrid Equivalence]
Two polynomial-like (or anti-polynomial-like) mappings $f : U^{\prime} \rightarrow U$ and $g : V^{\prime} \rightarrow V$ are \emph{hybrid equivalent} if there exists a quasi-conformal homeomorphism $\phi : U^{''} \rightarrow V^{''}$ between neighborhoods $U^{''}$ and $V^{''}$ of $K(f)$ and $K(g)$ respectively, such that $\phi \circ f = g \circ \phi$ whenever both sides are defined, and $\overline{\partial}\phi = 0$ almost everywhere in $K(f)$.
\end{definition}

\begin{theorem}[Straightening Theorem]\label{Anti-holomorphic Straightening Theorem}
Any polynomial-like (respectively anti-polynomial-like) mapping $g : U^{\prime} \rightarrow U$ is hybrid equivalent to a holomorphic (respectively anti-holomorphic) polynomial $P$ of the same degree. Moreover, if $K(g)$ is connected, then $P$ is unique up to affine conjugacy.
\end{theorem} 

\begin{remark}
We define the degree of $g$ as the number of pre-images of any point, so it is always positive. Hence, for an anti-holomorphic map $g$, $d$ is the degree (in the classical sense) of the proper holomorphic map $g^{\ast}  \colon U \to V^{\ast}$ which is the complex conjugate of $g$.
\end{remark}

\begin{definition}[Renormalization and Straightening] 
We say $f_c$ is \emph{renormalizable} if there exist $U^{\prime}_c$, $U_c$ (containing the critical point $0$), and $k > 1$ such that $f_c^{\circ k} : U^{\prime}_c \rightarrow U_c$ is unicritical-like, and has a connected filled Julia set.

Such a mapping $f_c^{\circ k} : U^{\prime}_c \rightarrow U_c$  is called a \emph{renormalization} of $f_c$, and $k$ is called its \emph{period}. 

By the straightening theorem, there exists a unique monic centered holomorphic or anti-holomorphic unicritical polynomial $P$ hybrid equivalent to $f_c^{\circ k} : U^{\prime}_c \rightarrow U_c$, up to affine conjugacy. We call $P$ the \emph{straightening} of the renormalization.
\end{definition}

Take $c_0 \in \mathcal{M}_d^*$ such that $0$ is a periodic point of period $k>1$ of $f_{c_0}$; i.e., $c_0$ is a center (of a hyperbolic component of $\mathcal{M}_d^*$) of period $k$. Let $\lambda_{0} = \lambda(f_{c_0})$ be the rational lamination of $f_{c_0}$. Define the combinatorial renormalization locus $\mathcal{C}(c_0)$ as follows:
\begin{center}
$\mathcal{C}(c_0) = \lbrace c \in \mathcal{M}_d^* : \lambda(f_c) \supset \lambda_{0} \rbrace$.
\end{center}

Since a rational lamination is an equivalence relation on $\mathbb{Q}/\mathbb{Z}$, it is a subset of $\mathbb{Q}/\mathbb{Z} \times \mathbb{Q}/\mathbb{Z}$, and hence, subset inclusion makes sense. By definition, for $c \in \mathcal{C}(c_0)$, the external rays of $\lambda_0$-equivalent angles for $f_c$ land at the same point. Hence those rays divide $K_c$ into `\emph{fibers}'. Let $K$ be the fiber containing the critical point $0$. Then $f_c^{\circ k}(K) = K$. We say that $f_c$ is $c_0$-\emph{renormalizable} if there exists a (holomorphic or anti-holomorphic) unicritical-like restriction $f_c^{\circ k} : U^{\prime}_c \rightarrow U_c$ such that the filled Julia set is equal to $K$. Let the renormalization locus $\mathcal{R}(c_0)$ with combinatorics $\lambda_0$ be:
\begin{center}
$\mathcal{R}(c_0) = \lbrace c \in \mathcal{C}(c_0) : f_c \hspace{1mm}$ is $c_0$-renormalizable$\rbrace$.
\end{center}

We call such a renormalization a $c_0$-\emph{renormalization} (see the definition of $\lambda_0$-\emph{renormalization} in \cite{IK} for a more general definition). We call $k$ the \emph{renormalization period}.

For the rest of this section, we fix such a $c_0$, and its rational lamination $\lambda_0 = \lambda(f_{c_0})$. For $c \in \mathcal{R}(c_0)$, let $P$ be the straightening of a $c_0$-renormalization of $f_c$. By the straightening theorem, $P$ is well-defined. When the renormalization period $k$ is even, then the $c_0$-renormalization is holomorphic. Hence, $P = f_{c^{\prime}}$ for some $c^{\prime} \in \mathcal{M}_d$. When $k$ is odd, the $c_0$-renormalization is anti-holomorphic, so $P = f_{c^{\prime}}$ for some $c^{\prime} \in \mathcal{M}_d^*$. In either case, we denote $c^{\prime}$ by $\chi_{c_0}(c)$ (here, we have tacitly fixed an external marking for our (anti-)polynomial-like maps so that the map $\chi_{c_0}$ is well-defined).

To relate our definition of $\chi_{c_0}$ with the general notion of straightening for holomorphic polynomials (as developed in \cite{IK}),we need to work with $P_{\overline{c},c}=f_{c}^{\circ 2}$. This allows us to embed $\lbrace f_c\rbrace_{c\in \mathbb{C}}$ in the family $\Poly(d^2)$ of monic centered polynomials of degree $d^2$. We will denote the real $2$-dimensional plane in which $\Poly(d^2)$ intersects the family $\lbrace P_{\overline{c},c}\rbrace_{c\in \mathbb{C}}$ by $L$. Since $P_{\overline{c_0},c_0}$ is a post-critically finite hyperbolic polynomial (of degree $d^2$) with rational lamination $\lambda_0$, we are now in the setting of renormalization and straightening maps defined over reduced mapping schemas. We refer the readers to \cite[\S 1]{IK} for these general notions.

The combinatorial renormalization locus
\[
 \mathcal{C}(\lambda_0) := \lbrace g \in \Poly(d^2): \lambda(g) \supset \lambda_0\rbrace,
\]
and the renormalization locus
\[
 \mathcal{R}(\lambda_0) = \lbrace g \in \mathcal{C}(\lambda_0): g \text{ is $\lambda_0$-renormalizable}\rbrace
\]
satisfy $\mathcal{C}(c_0) = \mathcal{C}(\lambda_0) \cap L$, and $\mathcal{R}(c_0) = \mathcal{R}(\lambda_0) \cap L$. There is a straightening map $\chi_{\lambda_0}: \mathcal{R}(\lambda_0)  \rightarrow \mathcal{C}(T(\lambda_0))$, where $\mathcal{C}(T(\lambda_0))$ is the fiber-wise connectedness locus of the family of monic centered polynomial maps over the reduced mapping scheme $T(\lambda_0)$ of $\lambda_0$. Following \cite{I1}, we will now describe the set $\mathcal{C}(T(\lambda_0))$.

When $k$ is even, $P_{\overline{c_0},c_0}$ has two disjoint periodic cycles each containing a single critical point of multiplicity $d$ (disjoint mapping scheme). This gives rise to two independent holomorphic unicritical-like maps (of degree $d$), and hence $\mathcal{C}(T(\lambda_0))= \mathcal{M}_d \times \mathcal{M}_d$, which is the fiber-wise connectedness locus of the family:

\begin{center}
$\lbrace P : \lbrace 0,1\rbrace \times \mathbb{C}\circlearrowleft; P(k,z)=(k, p_{a_k}(z)), p_{a_k}(z)= z^d+a_k, a_k\in \mathbb{C} \rbrace$
\\
$= \lbrace \left( p_{a_0}, p_{a_1}\right): a_0, a_1 \in \mathbb{C}\rbrace \cong \mathbb{C}^2.$
\end{center}

Now let $c\in \mathcal{R}(c_0)$. Every $c_0$-renormalization $f_c^{\circ k} : U_c^{\prime} \rightarrow U_c$ splits into two (holomorphic) unicritical-like maps (of degree $d$) $P_{\overline{c},c}^{\circ k/2}: U'_c\rightarrow U_c$ and $P_{\overline{c},c}^{\circ k/2}: f_c(U'_c)\rightarrow f_c(U_c)$ (after shrinking $U'_c$ and $U_c$ if necessary). Moreover, the former (holomorphic) unicritical-like restriction is anti-holomorphically conjugate to the latter one by $f_c^{\circ (k-1)}$ near the filled Julia sets (note that since $k$ is even, $f_c^{\circ (k-1)}$ is anti-holomorphic). Therefore, as $\lambda_0$-renormalization for $P_{\overline{c},c}$, we have two (holomorphic) unicritical-like maps of degree $d$ which are anti-holomorphically equivalent. After fixing an external marking for our polynomial-like maps, we conclude that the straightening of $P_{\overline{c},c}^{\circ k/2}: U'_c\rightarrow U_c$ and $P_{\overline{c},c}^{\circ k/2}: f_c(U'_c)\rightarrow f_c(U_c)$ are of the form $p_{c'}$ and $p_{\overline{c'}}$ (recall that $p_c(z)=z^d+c$). Therefore, modulo a fixed choice of external marking, for any $c\in \mathcal{R}(c_0)$ we have that $\chi_{\lambda_0}(P_{\overline{c},c})=(p_{c'},p_{\overline{c'}})$, for a unique $c^{\prime} \in \mathcal{M}_d$ (by the condition of having a connected filled Julia set). On the other hand, the $c_0$-renormalization $f_c^{\circ k} : U_c^{\prime} \rightarrow U_c$ is holomorphic and unicritical-like of degree $d$. By the definition of straightening, $\chi_{c_0}(f_c) = p_{c^{\prime}}$. 

Now let $k$ be odd. Then both the periodic critical points (of multiplicity $d$) of $P_{\overline{c_0},c_0}$ lie on the same cycle (bitransitive mapping scheme). For any $c\in \mathcal{R}(c_0)$, the quartic-like map $P_{\overline{c},c}^{\circ k}: U'_c\rightarrow P_{\overline{c},c}^{\circ k}(U'_c)=f_c^{\circ k}(U_c)$ can be written as the composition of the two unicritical holomorphic maps
\begin{align*}
 Q_1&: P_{\overline{c},c}^{\circ \frac{k-1}{2}}: U'_c\rightarrow f_c^{\circ (k-1)}(U'_c), &
 Q_2&: P_{\overline{c},c}^{\circ \frac{k+1}{2}}: f_c^{\circ (k-1)}(U'_c)\rightarrow f_c^{\circ k}(U_c),
\end{align*}
each of degree $d$. Hence it follows that the straightening is a composition of two degree $d$ unicritical polynomials. Therefore, $\mathcal{C}(T(\lambda_0))$ is the fiber-wise connectedness locus of the family:

\begin{align*}
\Poly(d\times d)
&= \lbrace P : \lbrace 0,1\rbrace \times \mathbb{C}\circlearrowleft; P(k,z)=(1-k, p_{a_k}(z)), 
\\
& \hspace{2cm} p_{a_k}(z)= z^d+a_k, a_k\in \mathbb{C} \rbrace
\\
&= \lbrace \left( p_{a_0}, p_{a_1}\right): a_0, a_1 \in \mathbb{C}\rbrace \cong \mathbb{C}^2.
\end{align*}

Identifying any $P \in \Poly(d\times d)$ with the composition $p_{a_1}\circ p_{a_0}$, we can view $\mathcal{C}(T(\lambda_0))$ as the connectedness locus of the family $\lbrace (z^d+a)^d+b\rbrace_{a, b \in \mathbb{C}}$. 

Now let $c\in \mathcal{R}(c_0)$, and $\chi_{\lambda_0}(P_{\overline{c},c})=(p_{a_0}, p_{a_1})$. Note that
\begin{align*}
  Q_2 \circ Q_1 &= P_{\overline{c},c}^{\circ k}: U'_c\rightarrow f_c^{\circ k}(U_c) \text{ and} \\
  Q_1 \circ Q_2 &= P_{\overline{c},c}^{\circ k}: f_c^{\circ (k-1)}(U'_c)\rightarrow f_c^{\circ (2k-1)}(U_c)
\end{align*}
are anti-holomorphically conjugate by $f_c$ near their filled Julia sets. Therefore, the straightenings of $Q_2\circ Q_1$ and $Q_1\circ Q_2$ are conjugate by an affine anti-holomorphic map. Hence, they satisfy $(p_{\overline{a_0}}, p_{\overline{a_1}})=(p_{a_1}, p_{a_0})$. Therefore, $\chi_{\lambda_0}(P_{\overline{c},c})= (p_{a_0}, p_{\overline{a_0}})$. Using the identification of $ \Poly(d\times d)$ with maps of the form $(z^d+a)^d+b$, we obtain that $\chi_{\lambda_0}(P_{\overline{c},c})= p_{a_0}\circ p_{\overline{a_0}}= f_{a_0}^{\circ 2}$, for a unique $a_0 \in \mathcal{M}_d^*$, once we have fixed an external marking for our (anti\mbox{-)}\nobreak\hspace{0pt}polynomial-like maps (by the condition of having a connected filled Julia set). On the other hand, the $c_0$-renormalization $f_c^{\circ k} : U^{\prime}_c \rightarrow U_c$ is anti-holomorphic and unicritical-like of degree $d$. By the definition of straightening, $\chi_{c_0}(f_c) = f_{a_0}$. 

The above discussion (along with our chosen identifications) shows that the maps $\chi_{c_0}$ and $\chi_{\lambda_0}$ are essentially the same on $\mathcal{R}(c_0)$.

Define
$$
\mathcal{M}(c_0) = \left\{\begin{array}{ll}
                    \mathcal{M}_d & \mbox{if}\ k\  \mbox{is even,}\\
                    \mathcal{M}_d^* & \mbox{if}\ k\ \mbox{is odd.}
                      \end{array}\right. 
$$

\begin{definition}[Straightening Map]
We call the map $\chi_{c_0} : \mathcal{R}(c_0) \rightarrow \mathcal{M}(c_0)$ as above the \emph{straightening map} for $c_0$.
\end{definition}

We call $\mathcal{C}(c_0)$ a \emph{baby multibrot set} when the renormalization period is even. Otherwise, we call it a \emph{multicorn-like set}. 

By the rotational symmetry of the multicorns, if the period $k$ is odd, then $\omega \chi_{c_0}, \omega^2 \chi_{c_0}, \cdots, \omega^d \chi_{c_0}$ are also straightening maps (with different internal/external markings), where $\omega = \exp({\frac{2 \pi i}{d+1}})$. In the sequel, we will always choose, and fix one of them.

\begin{definition}
We call a center $c_0 \in \mathcal{M}_d^*$ primitive if the closures of the bounded Fatou components of $f_{c_0}$ are mutually disjoint.
\end{definition}

With these preparations, we are now ready to state the main results from \cite{IK} applied to our setting. Strictly speaking, these theorems hold for the map $\chi_{\lambda_0}$, but we can apply them to the map $\chi_{c_0}$ since these two maps, suitably interpreted, agree on $\mathcal{R}(c_0)$.

\begin{theorem}[Injectivity]\label{injectivity}
The straightening map $\chi_{c_0} : \mathcal{R}(c_0) \rightarrow \mathcal{M}(c_0)$ is injective.
\end{theorem}

\begin{theorem}[Onto Hyperbolicity]\label{onto_hyp}
The image $\chi_{c_0}(\mathcal{R}(c_0))$ of the straightening map contains all the hyperbolic components of $\mathcal{M}(c_0)$.
\end{theorem}

\begin{theorem}[Compactness]\label{compactness}
If $c_0$ is primitive, then $\mathcal{C}(c_0) = \mathcal{R}(c_0)$, and it is compact.
\end{theorem}

These general theorems also imply that in the even period case, the straightening map from a baby multibrot-like set to the original multibrot set is a homeomorphism (at least in good cases). See \cite[Appendix~A]{I1} for a proof.

\begin{theorem}\label{even_homeo}
If $c_0$ is primitive, and the renormalization period is even, then the corresponding straightening map $\chi_{c_0} : \mathcal{R}(c_0) \rightarrow \mathcal{M}_d$ is a homeomorphism.
\end{theorem}

\section{A Continuity Property of Straightening Maps}\label{continuity}
According to \cite[Theorem~C]{IK}, the straightening map $\chi$, restricted to any hyperbolic component of $\mathcal{M}_d^*$, is a real-analytic homeomorphism. $\chi$ admits a homeomorphic extension to the boundaries of even period hyperbolic components under the conditions of Theorem~\ref{even_homeo}. In this section, we will study the corresponding property of $\chi$ when the renormalization period is odd. In fact, we will show that in this case, $\chi$ always extends as a homeomorphism to the boundaries of odd period hyperbolic components. 

We would like to thank John Milnor for bringing this question to our attention. In the special case $d=2$, this has been independently answered in \cite[\S 3]{BBM2}, and our treatment borrows heavily from their discussion. Since the terminology and the parameter spaces under consideration in \cite{BBM2} differ from ours, it is worthwhile to record the results here.

Let $H$ be a hyperbolic component of odd period $k$ of $\mathcal{M}_d^*$, and let $c_0$ be the center of $H$. Let $c \in H\setminus \lbrace c_0\rbrace$, and $z_0$ be the attracting periodic point of $f_c$ contained in the critical value Fatou component $U_c$.    

So, $f_c^{\circ k}(z_0) = z_0$. Let $\frac{\partial f_c^{\circ k}}{\partial \bar{z}}(z_0) = \lambda_c$. 

The multiplier $\frac{\partial f_c^{\circ 2k}}{\partial z}(z_0)$ is:
$$\frac{\partial f_c^{\circ 2k}}{\partial z}(z_0)=\frac{\partial f_c^{\circ k}}{\partial \overline{z}}(z_0)  \overline{\frac{\partial f_c^{\circ k}}{\partial \overline{z}}(z_0)}= \left| \frac{\partial f_c^{\circ k}}{\partial \overline{z}}(z_0) \right|^2= | \lambda_c |^2.$$

We will now associate a conformal invariant to $f_c$. In fact, the notion is similar to that of Ecalle height of a parabolic parameter. In our situation, there are two distinct critical orbits (for the second iterate $f_c^{\circ 2}$) converging to an attracting cycle. One can choose two representatives of these two critical orbits in a fundamental domain (in the critical value Fatou component), and consider their ratio under a holomorphic Koenigs coordinate. More precisely, if $\kappa_c : U_c\to \mathbb{C}$ is a Koenigs linearizing coordinate for the unique attracting periodic point of $f_c$ in $U_c$ with $\kappa_c (f_c^{\circ 2k}(z))=\vert\lambda_c\vert^2 \kappa_c(z)$, then we define an invariant
\begin{align*}
\rho_H(c)
&:=\frac{\kappa_c(f_c^{\circ k}(c))}{\kappa_c(c)}\;.
\end{align*} 

At the center $c_0$, we define $\rho_H(c_0)=0$.

This ratio is well-defined as the choice of Koenig's coordinate does not affect it, and hence is a conformal invariant of $f_c$. Moreover, $\vert \rho_H(c)\vert=\vert\lambda_c\vert$, so $\vert \rho_H(c)\vert\to +1$ as $c\to \partial H$. We will call it the \emph{Koenigs ratio}. It is easy to see that $\rho_H(c)$ agrees with the `critical value map' for odd period hyperbolic components introduced in \cite[\S 5]{NS}, so $\rho_H: H \to \mathbb{D}$ is a real-analytic $(d+1)$-fold branched covering branched only at the origin.

\begin{definition}[Internal Rays of Odd Period Components]
An \emph{internal ray} of an odd period hyperbolic component~$H$ of $\mathcal{M}_d^*$ is an
arc~$\gamma \subset H$ starting at the center such that
there is an angle~$\theta$ with~$\displaystyle \rho_{H}(\gamma) = \lbrace re^{2\pi i\theta}: r \in [0,1)\rbrace$.
\end{definition}

\begin{remark}
Since~$\rho_{H}$ is a~$(d+1)$-to-one map, an internal ray of~$H$
with a given angle is not uniquely defined. In fact, a hyperbolic component of odd period has~$(d+1)$ internal rays with
any given angle~$\theta$.
\end{remark}

Let $\widetilde{c}$ be a non-cusp parabolic parameter on the boundary of $H$. To understand the landing behavior of the internal rays, we will now relate the Koenigs ratio of $f_c$ to the critical Ecalle height of $f_{\widetilde{c}}$ as $c$ approaches $\widetilde{c}$. We have the following lemmas.

\begin{lemma}[Relation between Koenigs Ratio and Ecalle Height]\label{relating_invariants}
As $c$ in $H$ approaches a non-cusp parabolic parameter with critical Ecalle height $h$ on the boundary of $H$, the quantity $\frac{1-\rho_H(c)}{1-\vert\rho_H(c)\vert^2}$ converges to $\frac{1}{2}-2ih$.
\end{lemma}
\begin{proof}
Set $S_c(w)=\frac{\vert\lambda_c\vert^2(w-1)}{(\vert\lambda_c\vert^2-1)w}$. A direct computation shows that $S_c\circ\kappa_c(f_c^{\circ k}(c))-S_c\circ\kappa_c(c)=\frac{1-\rho_H(c)}{1-\vert\rho_H(c)\vert^2}$, and $\psi_{\widetilde{c}}^{\mathrm{att}}(f_c^{\circ k}(c))-\psi_{\widetilde{c}}^{\mathrm{att}}(c)=\frac{1}{2}-2ih$. Now using \cite[Theorem~1.2]{Ka}, we obtain the limiting relation between the two conformal invariants as $c$ approaches the parabolic parameter $\widetilde{c}$.
\end{proof}

The landing properties of the internal rays follow directly from the above lemma.

\begin{lemma}[Internal Rays Land]\label{Internal_Rays_Land}
The $d+1$ internal rays at angle $0$ land at the $d+1$ critical Ecalle height $0$ parameters on $\partial H$ (one ray on each parabolic arc). All other internal rays land at the cusp points on $\partial H$.
\end{lemma}

\begin{proof}
Let $\gamma$ be an internal ray at angle $\theta$, and $\widetilde{c}$ be an accumulation point of $\gamma$ on $\partial H$. Further assume that the critical Ecalle height of $f_{\widetilde{c}}$ is $h$. As $c$ approaches $\widetilde{c}$ (along $\gamma$), $\vert\rho_H(c)\vert$ goes to $+1$, and $\frac{1-\vert\rho_H(c)\vert e^{2\pi i\theta}}{1-\vert\rho_H(c)\vert^2}$ converges to $\frac{1}{2}-2ih$ (by Lemma~\ref{relating_invariants}). It follows that $\theta=0$. Bur for $\theta=0$, we have $\frac{1-\vert\rho_H(c)\vert}{1-\vert\rho_H(c)\vert^2}=\frac{1}{1+\vert\rho_H(c)\vert}\to\frac{1}{2}$ as $\vert\rho_H(c)\vert$ goes to $+1$, i.e., as $c$ goes to $\partial H$. This shows that the only accumulation point of the internal rays at angle $0$ are the critical Ecalle height $0$ parameters; i.e., these rays land there (note that there are only finitely many critical Ecalle height $0$ parameters on $\partial H$). On the other hand, the above argument shows that no internal ray at an angle $\theta\neq 0$ can accumulate at non-cusp parameters. Since $\overline{H}$ is compact, and $\partial H$ consists of $d+1$ parabolic arcs (of non-cusp parameters), and $d+1$ cusp points, it follows that every internal ray at angle $\theta$ different from $0$ lands at a cusp on $\partial H$. 
\end{proof}

Finally, note that the limiting relation between Koenigs ratio and Ecalle height obtained in Lemma~\ref{relating_invariants} holds uniformly for any hyperbolic component of odd period of $\mathcal{M}_d^*$. Since the straightening map $\chi$ preserves these conformal invariants, we have the following theorem.

\begin{theorem}[Homeomorphism between Closures of Odd Period Hyperbolic Components]\label{Homeo_Odd_Period}
Let $c_0$ be the center of a hyperbolic component of odd period of $\mathcal{M}_d^*$. Then $\chi: \mathcal{R}(c_0)\rightarrow \mathcal{M}_d^*$, restricted to the closure $\overline{H'}$ of any odd period hyperbolic component $H'\subset \mathcal{R}(c_0)$, is a homeomorphism.
\end{theorem}

\section{Discontinuity of Straightening Maps}\label{Discontinuity_of_The_Straightening_Map}

In this section, we will continue our study of straightening maps as developed in the previous sections, and will prove our main theorem on discontinuity of straightening maps in the odd period case.

It follows from an argument similar to Lemma~\ref{primitive} that the centers of odd period hyperbolic components of $\mathcal{M}_d^*$ are primitive. Using Theorem~\ref{compactness}, we conclude that:

\begin{corollary}\label{odd_compact}
If the renormalization period is odd, then $\mathcal{C}(c_0) = \mathcal{R}(c_0)$, and it is compact.
\end{corollary}

\begin{corollary}\label{reals_in_range}
If the renormalization period is odd, then the image $\chi_{c_0}(\mathcal{R}(c_0))$ of the straightening map contains the real part $\mathbb{R} \cap  \mathcal{M}_d^*$.
\end{corollary}

\begin{proof}
The proof is similar to \cite[Corollary~5.3]{I1}, where this has been proved in the quadratic case. For the degree $d$ case, one needs to use density of hyperbolicity of real polynomials proved in \cite{KSS,L4}. 
\end{proof}

Before giving the proof of the main theorem of this paper, we need to show the existence of hyperbolic components in the real part of $\mathcal{M}_d^*$. 

\begin{lemma}\label{3_in_real}
If $d$ is even, then at least one hyperbolic component of period $3$ of $\mathcal{M}_d^*$ intersects $\mathbb{R}\cap \mathcal{M}_d^*$.
\end{lemma}

\begin{proof}
Let $S_d$ be the set of all hyperbolic components of period $3$ of $\mathcal{M}_d^*$. By \cite[Theorem~1.3, Lemma~7.1]{MNS}, $\vert S_d \vert = d^2-1$ . Note that $\mathcal{M}_d^*$ has a complex conjugation symmetry; i.e., complex conjugation induces an involutive bijection on the set $S_d$. When $d$ is even, $\vert S_d \vert$ is an odd integer, and this implies that there must be some element $H_d^*$ in $S_d$ fixed by complex conjugation. Clearly, $H_d^*$ intersects the real line. Moreover, the center of $H_d^*$, and the critical Ecalle height $0$ parameter on the root arc of $\partial H_d^*$ are real, and a piece of $\mathbb{R}\cap \mathcal{M}_d^*$ converges to this critical Ecalle height $0$ parameter from the exterior of $\overline{H_d^*}$.
\end{proof}

\begin{proof}[Proof of Theorem~\ref{Straightening_discontinuity}]
Let $d$ be even, and $c_0$ be the center of a hyperbolic component of odd period $k$ ($\neq 1$) of $\mathcal{M}_d^*$.  We will assume that the map $\chi_{c_0} : \mathcal{R}(c_0) \rightarrow \mathcal{M}_d^*$ is continuous, and will arrive at a contradiction. 

By Corollary~\ref{odd_compact} and Theorem~\ref{injectivity}, $\mathcal{R}(c_0)$ is compact, and the map $\chi_{c_0}$ is injective. Since an injective continuous map from a compact topological space onto a Hausdorff topological space is a homeomorphism, it follows that $\chi_{c_0}$ is a homeomorphism from $\mathcal{R}(c_0)$ onto its range (we do not claim that $\chi_{c_0}(\mathcal{R}(c_0)) = \mathcal{M}_d^*$). It follows from Corollary~\ref{reals_in_range} (and by symmetry), that $\left(\mathbb{R}\cup \omega\mathbb{R}\cdots\cup \omega^d\mathbb{R}\right)\cap \mathcal{M}_d^* \subset \chi_{c_0}(\mathcal{R}(c_0))$. Moreover, by Theorem~\ref{onto_hyp}, $H_d^*\cup\omega H_d^*\cup\cdots\omega^{d} H_d^* \subset \chi_{c_0}(\mathcal{R}(c_0))$. 

Since $c_0$ is not of period $1$, there exists $i \in \{ 0, 1, \cdots, d \}$ such that $H^{\prime} :=\chi_{c_0}^{-1} \left(\omega^{i} H_d^*\right)$ does not intersect the real line or its $\omega$-rotates (but is contained in $\mathcal{R}(c_0)$). Recall that there exists a piece $\gamma$ of $\omega^{i} \mathbb{R}\cap \mathcal{M}_d^*$ that lies outside of $\overline{\omega^{i} H_d^*}$, and lands at the critical Ecalle height $0$ parameter on the root parabolic arc of $\partial\left(\omega^{i} H_d^*\right)$. By our assumption, $\chi_{c_0}$ is a homeomorphism; and hence the curve $\chi_{c_0}^{-1}\left(\gamma\right)$ lies in the exterior of $\overline{H^{\prime}}$, and lands at the critical Ecalle height $0$ parameter on the root arc of $\partial H^{\prime}$ (critical Ecalle heights are preserved by hybrid equivalences). Since $H^{\prime}$ does not intersect the real line or its $\omega$-rotates, this contradicts Theorem~\ref{umbilical_cord_wiggling}.

The hyperbolic component of period $3$ (intersecting the real line) does not play any special role in the above proof; in fact, there are infinitely many odd period hyperbolic components of $\mathcal{M}_d^*$ that intersect the real line. Our argument applies verbatim to any of these hyperbolic components, which proves that straightening maps are indeed discontinuous at infinitely many parameters.
\end{proof}

\begin{remark}
Although Theorem~\ref{Straightening_discontinuity} indicates that Tricorn-like sets appearing in other parameter spaces (of anti-holomorphic maps or holomorphic maps with real symmetry) are not homeomorphic to the original Tricorn, in many cases one can show that the locally connected topological models of the sets are homeomorphic. Such a strategy was successfully used in \cite{LLMM2} to show that the ``abstract connectedness locus" of a family of quadratic Schwarz reflection maps is homeomorphic to a suitable limb of the ``abstract Tricorn".
\end{remark}

\section{Are All Baby Multicorns Dynamically Distinct?}\label{geometric_discontinuity}

We proved in Section~\ref{Discontinuity_of_The_Straightening_Map} that for multicorns of even degree, the straightening map from a multicorn-like set based at an odd period (different from $1$) hyperbolic component to the multicorn is discontinuous at infinitely many parameters. However, we conjecture that there is a stronger form of discontinuity; i.e., any two multicorn-like sets (for multicorns of any degree) are dynamically different.

Recall that there are two important conformal invariants associated with every odd period non-cusp parabolic parameter; namely the critical Ecalle height, and the holomorphic fixed point index of the parabolic cycle. While critical Ecalle height is preserved by straightening maps, the parabolic fixed point index is in general not (since a hybrid equivalence does not necessarily preserve the external class of a polynomial-like map). In this section, we will prove that continuity of straightening maps would force the above two conformal invariants to be uniformly related along every parabolic arc. We will then look at some explicit Tricorn-like sets, and use the above information to show that the straightening maps between these Tricorn-like sets are discontinuous at infinitely many parameters on certain root parabolic arcs.

Let us now set up some notation. Let $H$ be a hyperbolic component of odd period $k$, $\mathcal{C}$ a parabolic arc of $\partial H$, $c_1:\mathbb{R}\to\mathcal{C}$ be the critical Ecalle height parametrization of $\mathcal{C}$, and $H'$ a hyperbolic component of period $2k$ bifurcating from $H$ across $\mathcal{C}$. Any straightening map $\chi$ restricted to $\overline{H}$ is a homeomorphism. Let $c_2:\mathbb{R}\to\chi(\mathcal{C})$ be the critical Ecalle height parametrization of $\chi(\mathcal{C})$. Since $\chi$ preserves critical Ecalle heights, we have $c_2=\chi\circ c_1$. For $h$ sufficiently large, $c_1(h) \in \mathcal{C} \cap \partial H'$. Let the fixed point index of  the unique parabolic cycle of $f_{c_1(h)}^{\circ 2}$ be $\tau$. Consider a curve $\gamma : \left[0,1\right] \to \overline{H'}$ with $\gamma(0)=c_1(h)$, and $\gamma(\left(0,1\right])\subset H'$. For $t\neq 0$, $f_{\gamma(t)}^{\circ 2}$ has two distinct $k$-periodic attracting cycles (which are born out of the parabolic cycle) with multipliers $\lambda_{\gamma(t)}$ and $\overline{\lambda_{\gamma(t)}}$. Then,

\begin{equation}\label{multiplier_index}
\frac{1}{1-\lambda_{\gamma(t)}} + \frac{1}{1-\overline{\lambda_{\gamma(t)}}} \longrightarrow \tau
\end{equation}
as $t\downarrow 0$. 

\begin{figure}[ht!]
\begin{center}
\includegraphics[scale=0.16]{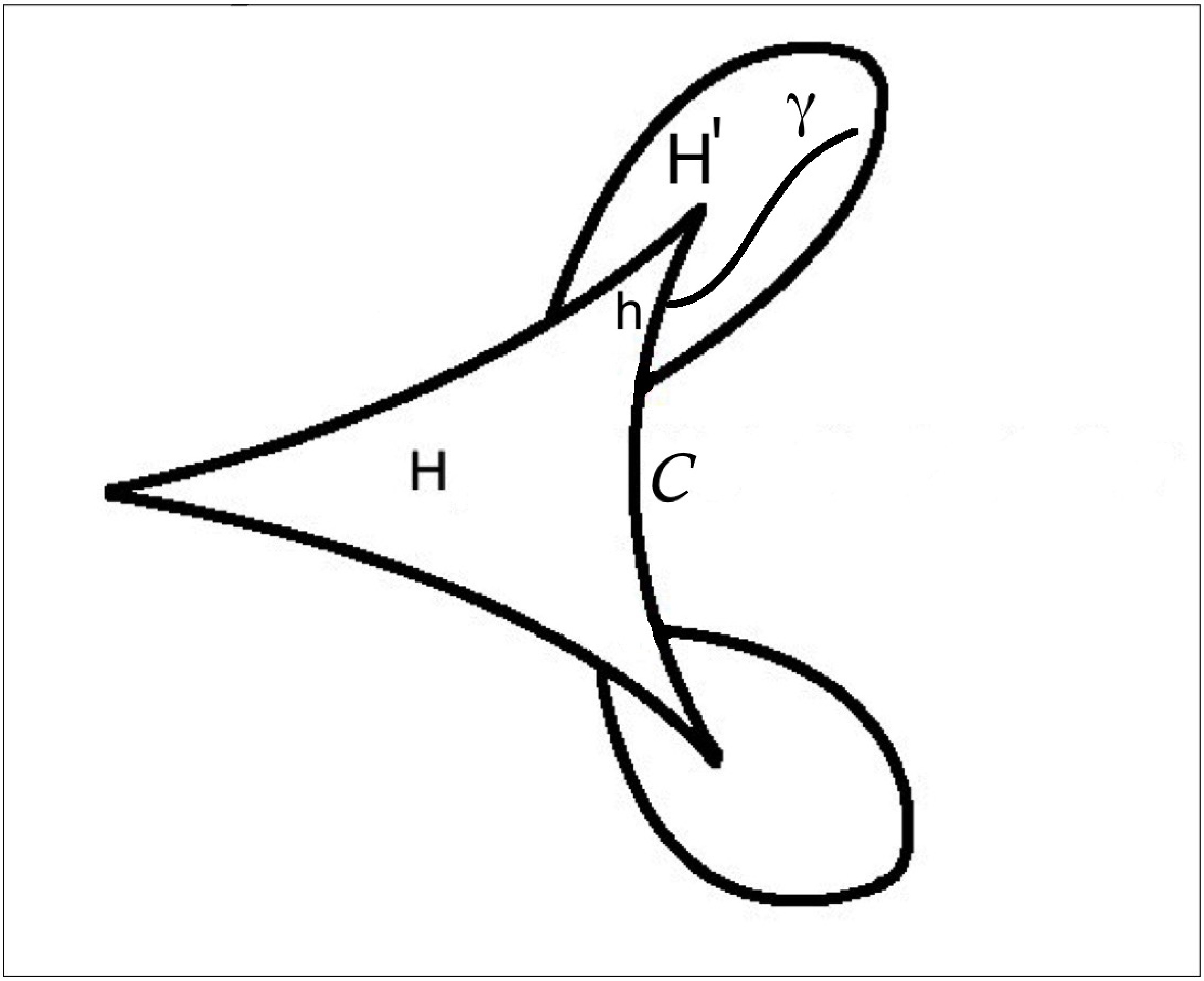}\hspace{1mm} \includegraphics[scale=0.16]{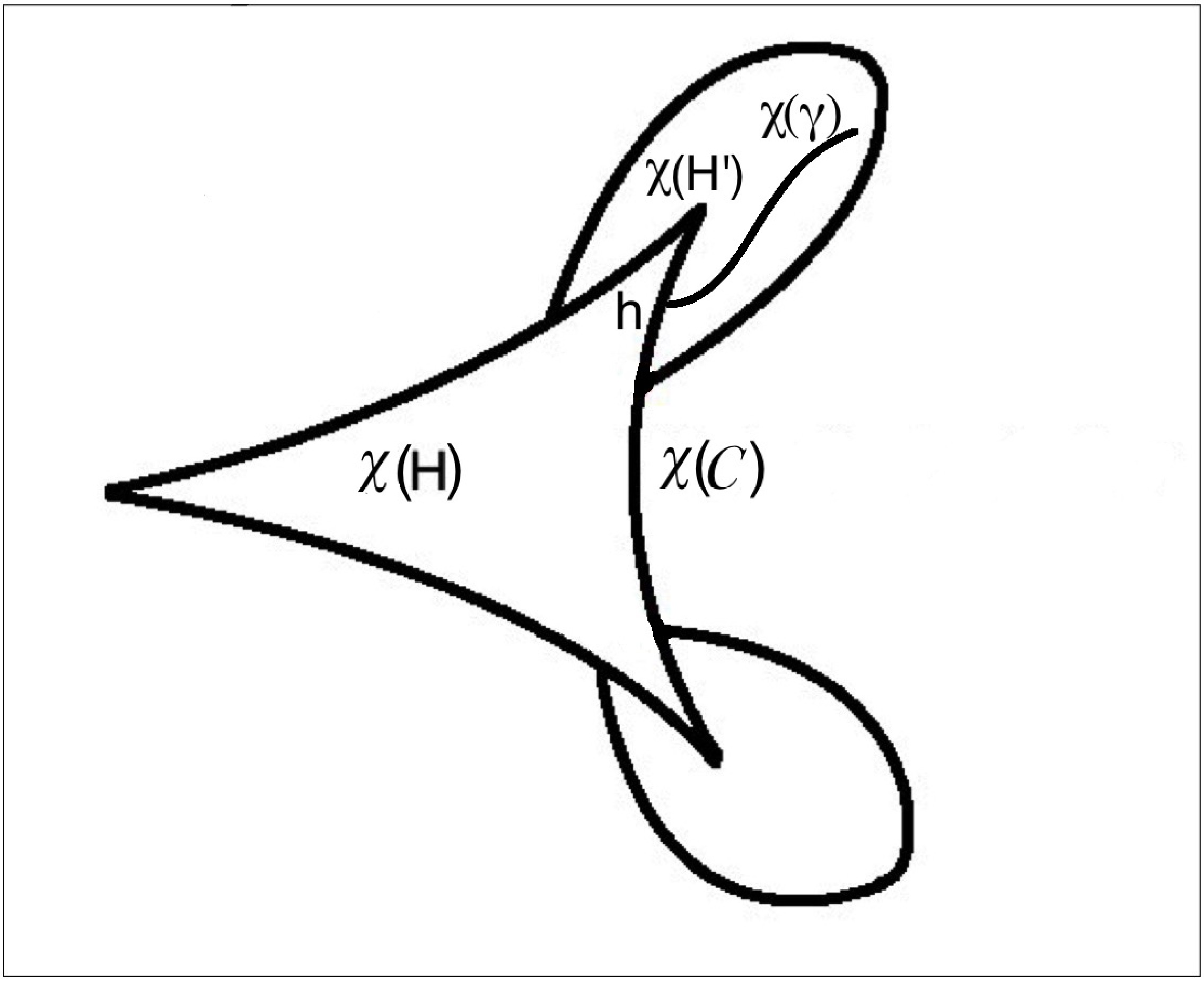}
\end{center}
\caption{Continuity of $\chi$ at $c_1(h)$ would force the fixed point indices of the parabolic cycles of $f_{c_1(h)}^{\circ 2}$ and $f_{c_2(h)}^{\circ 2}$ to be equal.}
\label{bifurcation_point_2}
\end{figure}

Continuity of the straightening map $\chi$ at $c_1(h)$ implies that $\displaystyle \lim_{t\downarrow 0} \chi(\gamma(t))=c_2(h)$ (compare Figure~\ref{bifurcation_point_2}).  But the multipliers of attracting periodic orbits are preserved by $\chi$. Therefore by Relation~\ref{multiplier_index}, the fixed point index of the parabolic cycle of $f_{c_2(h)}^{\circ 2}$ is also $\tau$. For any $h$ in $\mathbb{R}$, let us denote the fixed point index of the unique parabolic cycle of $f_{c_1(h)}^{\circ 2}$ (respectively of $f_{c_2(h)}^{\circ 2}$) by $\ind_{\mathcal{C}}(f_{c_1(h)}^{\circ 2})$ (respectively $\ind_{\chi(\mathcal{C})}(f_{c_2(h)}^{\circ 2})$). Since $c_1(h)$ was an arbitrary parameter on $\mathcal{C}\cap \partial H'$, continuity of the straightening map on $\mathcal{C}\cap \partial H'$ would imply that the index functions $\ind_{\mathcal{C}}: h\mapsto\ind_{\mathcal{C}}(f_{c_1(h)}^{\circ 2})$ and $\ind_{\mathcal{\chi(C)}}: h\mapsto\ind_{\chi(\mathcal{C)}}(f_{c_2(h)}^{\circ 2})$ agree on an interval of positive length. Since these functions are real-analytic, the identity theorem implies that these two functions must agree everywhere.

This seems extremely unlikely to hold, but we do not know how to rule out this possibility in general. However, the advantage of the preceding analysis is that the `dynamical' discontinuity of the straightening map would follow if we can prove that the functions $\ind_{\mathcal{C}}$ and $\ind_{\chi(\mathcal{C)}}$ disagree at a single point. We apply these observations to show that the Tricorn and the period $3$ Tricorn-like sets in the Tricorn are dynamically distinct. This is an indication of the fact that each Tricorn-like set carries its own characteristic geometry, which distinguishes it from other Tricorn-like sets.

More precisely, we have the following.

1) $\left(\mathcal{M}_2^*, \mathcal{C}\right)$; where $\mathcal{C}$ is the period $1$ parabolic arc intersecting the real line (of the Tricorn). The critical Ecalle height $0$ map on $\mathcal{C}$ is $f_{\frac{1}{4}}(z)=\overline{z}^2+\frac{1}{4}$. The parabolic fixed point index of $f_{\frac{1}{4}}^{\circ 2}$ is $\frac{1}{2}$.

2) $\left(\mathcal{T}_1, \mathcal{C}_1\right)$; where $\mathcal{T}_1$ is the period $3$ Tricorn-like set (in the Tricorn) intersecting the real line (equivalently, the renormalization locus $\mathcal{R}(\overline{z}^2+c_0)$, where $c_0$ is the airplane parameter), and $\mathcal{C}_1$ is the unique period $3$ root parabolic arc contained in $\mathcal{T}_1$. The critical Ecalle height $0$ map on $\mathcal{C}_1$ is $f_{-\frac{7}{4}}(z)=\overline{z}^2-\frac{7}{4}$. The parabolic fixed point index of $f_{-\frac{7}{4}}^{\circ 6}$ is $\frac{47}{98}$. One can compute this fixed point index as follows. The polynomial $z^2-\frac{7}{4}$ has two distinct fixed points, and the corresponding multipliers $\lambda_1$ and $\lambda_2$ are related by the equations $\lambda_1+\lambda_2=2$, and $\lambda_1\lambda_2=-7$. Let $\xi$ be the fixed point index of the parabolic fixed points of the third iterate of $z^2-\frac{7}{4}$. By the holomorphic fixed point formula (applied to the third iterate of $z^2-\frac{7}{4}$),
\begin{equation*}
3\xi + \frac{1}{1-\lambda_1^3}+\frac{1}{1-\lambda_2^3}=0.
\end{equation*}
A simple computation shows that $\xi=-\frac{2}{49}$. The parabolic fixed point index of $f_{-\frac{7}{4}}^{\circ 6}$ (which is the same as that of the sixth iterate of $z^2-\frac{7}{4}$) is given by $\frac{1+\xi}{2}=\frac{47}{98}$ (compare \cite[Lemma~12.9]{M1new}).

Under straightening maps, the parabolic arcs $\mathcal{C}$ and $\mathcal{C}_1$ correspond to each other. But the parabolic fixed point indices of their critical Ecalle height $0$ parameters are distinct. Hence the functions $\ind_{\mathcal{C}}$ and $\ind_{\mathcal{C}_1}$ differ at all but possibly a discrete set of real numbers. Therefore the corresponding straightening maps are discontinuous at all but possibly a discrete set of parameters on a sub-arc of the corresponding parabolic arcs. We summarize this observation in the following proposition.

\begin{proposition}[Dynamically Distinct Baby Tricorns]\label{3_1_dynamically_distinct}
The Tricorn and the period $3$ Tricorn-like sets in the Tricorn are dynamically distinct; i.e., they are \emph{not} homeomorphic via straightening map. More precisely, the corresponding straightening map is discontinuous at all but possibly a discrete set of parameters on certain sub-arc of $\mathcal{C}_1$.
\end{proposition}

In general, we conjecture that there is no `universal' formula for $\ind_{\mathcal{C}}$.

\begin{conjecture}[Baby Multicorns are Dynamically Distinct]\label{independent_babies}
Let $\mathcal{C}_1$ and $\mathcal{C}_2$ be two distinct parabolic arcs in $\mathcal{M}_d^*$ such that $\omega^i \mathcal{C}_1\neq \mathcal{C}_2$ for $i=1, 2, \cdots, d+1$. Then the functions $\ind_{\mathcal{C}_1}$ and $\ind_{\mathcal{C}_2}$ are not identically equal. Therefore two multicorn-like sets, which are not $\omega^i$-rotates of each other, are dynamically distinct; i.e., they are \emph{not} homeomorphic via straightening maps.
\end{conjecture}

\begin{remark}
1) Let us define a map $\xi : \overline{\mathcal{C}\cap\partial H'}\to \overline{\chi(\mathcal{C})\cap\partial \chi(H')}$ by sending the parabolic cusp (on $\overline{\mathcal{C}\cap\partial H'}$) to the parabolic cusp (on $\overline{\chi(\mathcal{C})\cap\partial \chi(H')}$), and sending the unique parameter (on $\mathcal{C}\cap\partial H'$) with parabolic fixed point index $\tau$ to the unique parameter (on $\chi(\mathcal{C})\cap\partial \chi(H')$) with parabolic fixed point index $\tau$. This definition makes sense because $\ind_{\mathcal{C}}(c_1^{-1}(\mathcal{C}\cap\partial H'))=\left[1,+\infty\right)$ (respectively $\ind_{\chi(\mathcal{C})}(c_2^{-1}(\chi(\mathcal{C})\cap\partial \chi(H')))=\left[1,+\infty\right)$), and $\ind_{\mathcal{C}}$ (respectively $\ind_{\chi(\mathcal{C})}$) is strictly increasing (in particular, injective) there (see Corollary~\ref{index_increasing}).

Using Lemma~\ref{even_rays_land}, it is not hard to see that $\xi$ is a continuous extension of $\chi\vert_{H'}$ to $\overline{\mathcal{C}\cap\partial H'}$. Therefore, Conjecture~\ref{independent_babies} implies that the maps $\xi$ and $\chi$ do not agree everywhere on $\mathcal{C}\cap\partial H'$.

2) Note that the original multicorn $\mathcal{M}_d^*$ has a $(d+1)-$fold rotational symmetry (compare Lemma~\ref{symmetry}). Conjecture~\ref{independent_babies} implies that the multicorn-like sets do not have any such symmetry.
\end{remark}

\end{document}